\documentclass{birkau}
\usepackage[T1]{fontenc}
\usepackage[latin9]{inputenc}
\usepackage{mathrsfs}
\usepackage{amsmath}
\usepackage{amsthm}
\usepackage{amssymb}
\usepackage{stmaryrd}
\usepackage{graphicx}

\makeatletter

\newcommand{\lyxmathsym}[1]{\ifmmode\begingroup\def\b@ld{bold}
  \text{\ifx\math@version\b@ld\bfseries\fi#1}\endgroup\else#1\fi}

\newenvironment{lyxlist}[1]
	{\begin{list}{}
		{\settowidth{\labelwidth}{#1}
		 \setlength{\leftmargin}{\labelwidth}
		 \addtolength{\leftmargin}{\labelsep}
		 }}
	{\end{list}}
\theoremstyle{remark}
\newtheorem*{notation*}{\protect\notationname}
\theoremstyle{plain}
\newtheorem{thm}{\protect\theoremname}[section]
\theoremstyle{definition}
\newtheorem{defn}[thm]{\protect\definitionname}
\theoremstyle{remark}
\newtheorem{rem}[thm]{\protect\remarkname}
\theoremstyle{plain}
\newtheorem{lem}[thm]{\protect\lemmaname}
\theoremstyle{definition}
\newtheorem{example}[thm]{\protect\examplename}
\theoremstyle{plain}
\newtheorem{prop}[thm]{\protect\propositionname}
\theoremstyle{plain}
\newtheorem{cor}[thm]{\protect\corollaryname}
\theoremstyle{plain}
\newtheorem{question}[thm]{\protect\questionname}

\usepackage{amsmath,amssymb,url}
\usepackage{enumitem} 

\numberwithin{equation}{section}

\theoremstyle{plain}

\theoremstyle{definition}

\DeclareMathOperator{\At}{At}
\DeclareMathOperator{\Trim}{Trim}

\makeatother

\providecommand{\corollaryname}{Corollary}
\providecommand{\definitionname}{Definition}
\providecommand{\examplename}{Example}
\providecommand{\lemmaname}{Lemma}
\providecommand{\notationname}{Notation}
\providecommand{\propositionname}{Proposition}
\providecommand{\questionname}{Question}
\providecommand{\remarkname}{Remark}
\providecommand{\theoremname}{Theorem}

\begin{document}
\title{Stone space partitions indexed by a poset}
\author[A. B. Apps]{Andrew B. Apps}

\address{Independent researcher, St Albans, UK}
\email{andrew.apps@apps27.co.uk}
\thanks{I gratefully acknowledge the use of Cambridge University's library facilities.}
\begin{abstract}
Stone space partitions $\{X_{p}\mid p\in P\}$ satisfying conditions
like $\overline{X_{p}}=\bigcup_{q\leqslant p}X_{q}$ for all $p\in P$,
where $P$ is a poset or PO system (poset with a distinguished subset),
arise naturally in the study both of primitive Boolean algebras and
of $\omega$-categorical structures. A key concept for studying such
partitions is that of a $p$-trim open set which meets precisely those
$X_{q}$ for which $q\geqslant p$; for Stone spaces, this is the
topological equivalent of a pseudo-indecomposable set. This paper
develops the theory of infinite partitions of Stone spaces indexed
by a poset or PO system where the trim sets form a neighbourhood base
for the topology. We study the interplay between order properties
of the poset/PO system and topological properties of the partition,
examine extensions and completions of such partitions, and derive
necessary and sufficient conditions on the poset/PO system for the
existence of the various types of partition studied. We also identify
circumstances in which a second countable Stone space with a trim
partition indexed by a given PO system is unique up to homeomorphism,
subject to choices on the isolated point structure and boundedness
of the partition elements. One corollary of our results is that there
is a partition $\{X_{r}\mid r\in[0,1]\}$ of the Cantor set such that
$\overline{X_{r}}=\bigcup_{s\leqslant r}X_{s}\text{ for all }r\in[0,1]$. 
\end{abstract}
\maketitle
\section{Introduction}

For a set $P$ with a transitive relation $\prec$, Ziegler~\cite[Part~II,~\S1.C,D]{FlumZiegler}
introduced the idea of a \emph{good partition }$\{X_{p}\mid p\in P\}$
of a $T_{3}$-topological space, where each $X_{p}$ is either discrete
or has no isolated points, and: 

\begin{equation}
\text{\ensuremath{\overline{X_{p}}=\bigcup_{q\preccurlyeq p}X_{q}\text{ for all }p\in P}}\label{eq:goodpartn}
\end{equation}

\emph{Finitely based }partitions of a Stone space $X$, defined as
partitions which can be refined to a good partition $\{X_{p}\mid p\in P\}$
of $X$ for a finite poset $P$ such that each $X_{p}$ has only finitely
many isolated points, arise naturally in the study of $\omega$-categorical
structures. We provide the relevant $\omega$-categorical background
at the end of this introduction; whilst this has provided the motivation
and some of the concepts for the current study, it is not necessary
for an understanding of the rest of this paper.

A \emph{PO system }is a poset $P$ with a distinguished subset $P_{1}$;
equivalently, $P$ has a transitive anti-symmetric relation $<$ with
$P_{1}=\{p\in P\mid p<p\}$. A~finitely based partition of a Stone
space $X$ can be further refined to yield a ``$P$-partition''
$\{X_{p}\mid p\in P\}$, where $P$ is a finite PO system with $P^{d}\subseteq P_{\min}$
(writing $P^{d}=P-P_{1}$ and $A^{\prime}$ for the derived set of
$A$), such that:

\begin{equation}
\begin{array}{c}
X_{p}^{\prime}=\bigcup_{q<p}X_{q}\text{ for all }p\in P\end{array}\label{eq:finbase}
\end{equation}

Apps~\cite[Theorem~C,~taking~the~group~actions~to~be~trivial]{AppsBP}
showed that for any such $P$, there is a compact Stone space $X$
that admits a partition $\mathscr{X}=\{X_{p}\mid p\in P\}$ satisfying~(\ref{eq:finbase})
and such that $|X_{p}|=1\text{ for }p\in P^{d}$, and that $\{X,\mathscr{X}\}$
is unique up to homeomorphisms respecting the partition. 

This result is also a corollary of the work of Pierce~\cite[Theorems 4.3 and 4.6]{Pierce}.
Pierce studied compact Stone spaces $X$ of ``finite type'', which
have only finitely many subsets invariant under homeomorphisms of
$X$. These subsets form a finite ``topological Boolean algebra''
$\mathscr{A}$, whose atoms can be labelled by a finite PO system
$P$ so as to form a complete $P$-partition of $X$ satisfying (\ref{eq:finbase}),
with $|X_{p}|$ taking values in $\mathbb{N}_{+}$ for $p\in P^{d}$.
Moreover, any finite topological Boolean algebra generated by its
closed elements can be embedded in a Stone space such that its atoms
form a complete $P$-partition of the space.

Stone spaces of finite type are a special type of primitive space.
An element $A$ of a Boolean ring $R$ is \emph{pseudo-indecomposable
(p.i.)} if for all $B\in(A)$, either $(B)\cong(A)$ or $(A-B)\cong(A)$
, where $(A)=\{B\in R\mid B\subseteq A\}$; and $R$ and its Stone
space are \emph{primitive }if every element of $R$ is the disjoint
union of finitely many p.i.\ elements. Hanf~\cite{Hanf} showed
that a primitive Boolean algebra $R$ such that $1_{R}$ is p.i.\ is
determined up to isomorphism by a PO system ``structure diagram''
$P$. Hansoul~\cite{Hansoul} further showed that with any countable
primitive $R$, we can associate a natural (and usually infinite)
complete partition of its Stone space $X$, indexed by the filters
on $P$, whose elements are invariant under homeomorphisms of $X$.

This paper develops the theory of infinite $P$-partitions satisfying
properties like~(\ref{eq:goodpartn}) if $P$ is a poset or~(\ref{eq:finbase})
if $P$ is a PO system, focussing on $P$ and the partition rather
than the underlying Stone space or Boolean ring. We consider conditions
on $P$ for the existence and uniqueness of such partitions, and the
interplay between order properties of $P$ and topological properties
of $P$-partitions. This provides a unifying theory to underpin both
the finitely based partitions arising for $\omega$-categorical structures
and the structure and rank diagrams associated with primitive Boolean
algebras. The main departures from existing work are firstly to study
$P$-partitions of a dense subset of a space rather than of the whole
space, as complete $P$-partitions of a Stone space do not exist for
a general $P$; and secondly to work with Boolean rings, whose Stone
spaces may not be compact, rather than Boolean algebras.

The key concept developed is that of a $p$\emph{-trim set }$A$ for
$A\in R$ (the Boolean ring underlying the Stone space), satisfying
the conditions $A\cap X_{q}\neq\emptyset$ iff $q\geqslant p$ and
(if $P$ is a PO system) $|A\cap X_{p}|=1$ if $p\in P^{d}$; this
provides a topological equivalent to the algebraic concept of a pseudo-indecomposable
set. We will primarily be interested in $P$-partitions $\{X_{p}\mid p\in P\}$
of a (dense) subset $X$ of a Stone space $W$ for which every element
of $W$ has a neighbourhood base of trim sets. This strengthens~(\ref{eq:goodpartn})~or~(\ref{eq:finbase})
by adding conditions on local properties of the partition, with elements
of $X_{p}$ dividing into \emph{clean points,} which have a $p$-trim
neighbourhood, and \emph{limit points}. We will consider \emph{trim
partitions}, where every element of $X_{p}$ is a clean point,\emph{
semi-trim partitions,} where clean points are dense in $X_{p}$ for
relevant $p$, and \emph{complete partitions, }where $X$ is the whole
Stone space $W$.

The paper is structured as follows. In Section~\ref{sec:Definitions-and-Basic}
we define and establish some basic properties of the different types
of partitions, with examples, linking properties of the partition
such as existence, compactness and isolated point structure with order
properties of the poset or PO system. Our initial definitions relate
to any topological space, but Sections~\ref{sec:Characterising-partitions-by}~to~\ref{sec:Uniqueness-Conditions}
focus on Stone spaces, particularly on those arising from a countable
Boolean ring, which we term \emph{$\omega$-Stone spaces}. 

In Section~\ref{sec:Characterising-partitions-by} we show that trim
sets are p.i., and that trim $P$-partitions of a Stone space are
effectively equivalent to structure diagrams as defined by Hanf~\cite{Hanf};
and hence that a countable Boolean ring is primitive iff its Stone
space admits a trim partition.

In Section~\ref{sec:Completion} we examine the criteria under which
a semi-trim $P$-partition can be extended to a $Q$-partition, where
$P,Q$ are posets or PO systems such that $P\subseteq Q$. We consider
the \emph{chain completion} and \emph{clean interior} of a semi-trim
partition and their inter-relationship, and show:

\paragraph{Completion (Corollary~\ref{Completion corollary}, weak version)}

Let $P$ be a countable poset and $W$ an $\omega$-Stone space. Then
any trim $P$-partition of $W$ can be extended to a complete semi-trim
$\overline{P}$-partition of $W$, where $\overline{P}$ is the chain
completion of $P$.

In Section~\ref{sec:Existence-Conditions} we show how to construct
trim $P$-partitions satisfying given isolated point and compactness
criteria, and obtain:

\paragraph{Complete trim partitions (Theorem~\ref{Trim nec condition}(i)) }

A trim $P$-partition of a Stone space is complete iff $P$ satisfies
the ascending chain condition.

\paragraph{Existence (Theorem~\ref{Existence theorem})}

Let $P$ be a countable poset or PO system. Then:
\begin{lyxlist}{00.00.0000}
\item [{(i)}] There is always an $\omega$-Stone space which admits a trim
$P$-partition;
\item [{(ii)}] There is an $\omega$-Stone space which admits a complete
semi-trim $P$-partition iff $P$ is $\omega$-complete (with an additional
``separation'' condition on $P^{d}$ if $P$ is a PO system).
\end{lyxlist}
In Section~\ref{sec:Uniqueness-Conditions} we consider when these
partitions are unique and obtain:

\paragraph{Uniqueness (Theorem~\ref{uniqueness thm}) }

Let $P$ be a countable PO system, and $L$ a suitable subset of $P$.
Then there is an $\omega$-Stone space $W$ which admits a trim $P$-partition
$\mathscr{X}=\{X_{p}\mid p\in P\}$, with $\{W,\mathscr{X}\}$ unique
up to homeomorphisms respecting $\mathscr{X}$, such that (i) $\overline{\bigcup_{p\in L}X_{p}}$
is compact; (ii) $\overline{X_{p}}$ is non-compact for all $p\notin L$;
(iii) $|X_{p}|$ takes predetermined values in $\mathbb{N}_{+}$ for
$p\in L_{\min}\cap P^{d}$.

A corollary of these results is that there is a partition $\{X_{r}\mid r\in[0,1]\}$
of the Cantor set (and therefore also of the unit interval) such that
$X_{0}$ is homeomorphic to the Cantor set and $\overline{X_{r}}=\bigcup_{s\leqslant r}X_{s}\text{ for all }r\in[0,1]$.

Finally in Section~\ref{sec:Closure-algebras-and} we classify partitions
associated with closure algebras generated by a single open set in
a topological space. We also show that there is an open subset $A$
of the Cantor set such that the closure algebra generated by $A$
mirrors the Rieger-Nishimura lattice, and such that each ``non-maximal''
atom within this closure algebra can independently be discrete or
have no isolated points.

We will use some of these results in a future paper to develop a new
type of Boolean construction, which in turn will yield some interesting
new examples of $\omega$-categorical structures.

\subsection{Background: $\omega$-categorical structures and $P$-partitions}

A countable structure $M$ for a language $L$ is \emph{$\omega$-categorical
}if $M$ is determined up to isomorphism within the class of countable
$L$-structures by its first order properties.

If $M$ is a structure and $R$ is a Boolean ring with Stone space
$X$, then the \emph{Boolean power of $M$ by $R$}, denoted $M^{R}$,
is the set of continuous functions from $X$ to $M$ with compact
support, where $M$ is given the discrete topology. If $M$ is an
$\omega$-categorical structure and $R$ is a countable Boolean ring
with only finitely many atoms, then $M^{R}$ is $\omega$-categorical~\cite{WWBP}.

More generally, if $M$ is a structure, $\{M_{i}\mid i\in I\}$ are
substructures of $M$, $X$ is the Stone space of a Boolean ring $R$,
and $\{C_{i}\mid i\in I\}$ are closed subsets of $X$, we can define
the \emph{filtered Boolean power}
\[
\Gamma=\{f\in M^{R}\mid\text{\ensuremath{C_{i}f\subseteq M_{i}\text{ for all \ensuremath{i\}}}}}.
\]

If $M$ and $I$ are finite and the system $\{X,\{C_{i}\mid i\in I\}\}$
is $\omega$-categorical, then $\Gamma$ is $\omega$-categorical~\cite{Schmerl}.

Macintyre and Rosenstein~\cite{MacRos} showed that an $\omega$-categorical
ring without nilpotent elements is isomorphic to a filtered Boolean
power
\begin{center}
$\{f\in E^{R}\mid\text{\ensuremath{C_{i}f\subseteq E_{i}\text{ for \ensuremath{i\leqslant n\}}}}}$,
\par\end{center}

where $E$ is a finite field; $E_{1},\ldots,E_{n}$ are subfields
of $E$; $C_{1},\ldots,C_{n}$ are closed subsets of the Stone space
$X$ of the Boolean ring $R$; and the system $\{X,\{C_{i}\mid i\leqslant n\}\}$
is $\omega$-categorical. Letting $J_{i}$ be the ideal of $R$ corresponding
to $C_{i}$ under the Stone correspondence, they further showed that
the system $(R,J_{1},\ldots,J_{n})$ is $\omega$-categorical iff
$\langle J_{1},\ldots,J_{n}\rangle$ is a finite sub-algebra of the
set of all ideals of $R$, viewed as a Heyting algebra, and $R/J$
has only finitely many atoms for each $J\in\langle J_{1},\ldots,J_{n}\rangle$.
Further studies of this condition may be found in~\cite{Alaev,Palchunov,Touraille}.
Letting $\mathscr{X}$ be the partition of $X$ consisting of the
atoms of the Boolean subring of $2^{X}$ generated by $\{C_{1},\ldots,C_{n}\}$,
it can be readily shown that this condition is equivalent to $\mathscr{X}$
being finitely based.

Apps~\cite{AppsBP} investigated when a finite extension $\Delta$
of $H^{R}$ is $\omega$-categorical, where $H$ is a finite perfect
indecomposable group and $R$ is a countable Boolean algebra with
finitely many atoms. $\Delta$ induces a finite subgroup $G$, say,
of $\mathrm{Aut}(R)$, the group of automorphisms of $R$, which also
acts on $X$, the Stone space of $R$. For each subgroup $K$ of $G$,
let $V(K)=\{x\in X\mid G_{x}=K\}$, where $G_{x}$ is the stabiliser
$\{g\in G\mid xg=x\}$ in $G$ of $x$. We can then associate with
$G$ the partition $\mathscr{V}$$(G)=\{V(K)\mid K\leqslant G,V(K)\neq\emptyset\}$.
Apps~\cite[Theorem~A~and~Corollary~B1]{AppsBP} showed that $\Delta$
is $\omega$-categorical iff $\mathscr{V\mathrm{(}\mathit{G}\mathrm{)}}$
is finitely based: in particular, the $\omega$-categoricity of $\Delta$
depends only on the topological properties of $\mathscr{V\mathrm{(}\mathit{G}\mathrm{)}}$. 

This situation also gives rise to a system $[P,G,\{L_{p}\mid p\in P\}]$,
where $P$ is a PO system acted on by $G$ and $L_{p}\leqslant G$
for each $p\in P$. Apps~\cite[Theorem~C]{AppsBP} further showed
that for any such system satisfying the relevant consistency conditions,
there is a unique Boolean representation $(X,\mathscr{X})$, where
$X$ is a compact $\omega$-Stone space and $\mathscr{X}=\{X_{p}\mid p\in P\}$
is a partition of $X$ satisfying (\ref{eq:finbase}) such that $|X_{p}|=1\text{ for }p\in P^{d}$,
together with an action of $G$ on $X$ such that (i) $X_{p}g=X_{pg}$
for $p\in P$ and $g\in G$; (ii) if $x\in X_{p}$, then $G_{x}=L_{p}$.

\subsection{Notation; poset definitions}
\begin{notation*}
If $A$ and $B$ are sets, we write $A\subseteq B$ to denote that
$A$ is contained in $B$, $A-B$ for $\{x\in A\mid x\notin B\}$,
$A\dotplus B$ for the disjoint union of $A$ and $B$, $|A|$ for
the cardinal number of $A$, $\mathbb{N}$ for $\{0,1,2,\ldots\}$,
and $\mathbb{N}_{+}$ for $\{1,2,3,\ldots\}$.
\end{notation*}
We will make frequent use throughout of the Stone correspondence between
Boolean rings (which may or may not have an identity) and locally
compact totally disconnected Hausdorff spaces, under which Boolean
ring elements correspond to compact open subsets of the Stone space,
and atoms of the Boolean ring correspond to isolated points of the
Stone space. The Boolean ring has an identity iff its Stone space
is compact, and is countable iff its Stone space is second countable
(i.e.\ the topology has a countable base of open sets). We will use
the term \emph{$\omega$-Stone space }to denote the Stone space of
a countable Boolean ring.
\begin{notation*}
If $\{X_{p}\mid p\in P\}$ is a partition of a subset $X$ of $W$,
so that each $X_{p}\neq\emptyset$, we will write $\{X_{p}\mid p\in P\}^{*}$
for the partition $\{\{X_{p}\mid p\in P\},W-X\}$ of $W$, with the
convention that $W-X$ may be the empty set. The asterisk will be
omitted where a complete partition of $W$ is intended (i.e.\ where
$X=W$). 

For $Q\subseteq P$, we write $X_{Q}$ for $\bigcup_{p\in Q}X_{p}$.

If $\{X_{p}\mid p\in P\}$ and $\{Y_{q}\mid q\in Q\}$ are partitions
of $W$, we say that $\{Y_{q}\}$ \emph{refines }$\{X_{p}\}$ if for
each $q\in Q$, we can find $p\in P$ such that $Y_{q}\subseteq X_{p}$.

If $X$ is a topological space and $S\subseteq X$, we write $S^{\prime}$
for the \emph{derived set }of $S$, namely the set of all limit points
of $S$, and $\overline{S}$ for the closure of $S$. 

If $R$ is a Boolean ring and $A\in R$, we write $(A)$ for $\{B\in R\mid B\subseteq A\}$,
the ideal of $R$ generated by $A$; $0$ for the zero element; and
$1_{R}$ for the multiplicative identity if $R$ is a Boolean algebra.
\end{notation*}
\begin{defn}
We recall that a \emph{PO system }is a set $P$ with an anti-symmetric
transitive relation $<$; equivalently, it is a poset with a distinguished
subset $P_{1}=\{p\in P\mid p<p\}$. We write $p\leqslant q$ to mean
$p<q$ or $p=q$; and $P^{d}$ for $\{p\in P\mid p\nless p\}$. If
$P$ is a poset, we write $p\lneqq q$ if $p\leqslant q$ and $p\neq q$.

Let $P$ be a partially ordered set or PO system and let $Q\subseteq P$.
We write $Q_{\downarrow}=\{p\in P\mid p\leqslant q\text{ for some }q\in Q\}$.

We recall that $Q$ is an \emph{upper} (respectively \emph{lower})
subset of $P$ if for all $q\in Q$, if $r\geqslant q$ (respectively
$r\leqslant q$) then $r\in Q$.

$P$ is \emph{$\omega$-complete }if every $\omega$-chain of elements
$p_{1}\leqslant p_{2}\leqslant p_{3}\leqslant\cdots$ has a supremum
(least upper bound) in $P$, and $P$ has the \emph{ascending chain
condition }(\emph{ACC}) if it has no strictly increasing chain of
elements $p_{1}\lneqq p_{2}\lneqq p_{3}\lneqq\cdots$.

$Q$ is an \emph{antichain }if for any $p\in Q$ and $q\in Q$ with
$p\neq q$, $p\nleqslant q$.

$P_{\min}$ and $P_{\max}$ will denote the sets of minimal and maximal
elements of $P$ respectively.

We will say that $F\subseteq P$ is a \emph{finite foundation} of
$Q$ if $F$ is a finite subset of $Q_{\downarrow}$ such that for
all $r\in Q_{\downarrow}$ we can find $p\in F$ such that $p\leqslant r$.

We write $P_{\Delta}$ for $\{p\in P\mid\{p\}\text{ has a finite foundation}\}$.

We will say that a subset $F$ of $Q$ is a \emph{finite ceiling }of
$Q$ if $F$ is finite and for all $q\in Q$, we can find $r\in F$
such that $q\leqslant r$.

We recall that $J\subseteq P$ is an \emph{ideal} of $P$ if $J$
is a lower subset such that for all $x,y\in J$, we can find $z\in J$
such that $x\leqslant z$ and $y\leqslant z$.
\end{defn}

\begin{rem}
\label{rem:Ideals}If $P$ is a countable poset, a routine argument
shows that the ideals of $P$ have form $J=\{p\in P\mid p\leqslant q_{n}\text{ for some }n\}$,
where $\{q_{n}\mid n\geqslant1\}$ is an increasing sequence in $P$
and $q_{1}$ can be any element of $J$; and that if $J$ has a supremum
then $\sup_{P}J=\sup_{P}\{q_{n}\mid n\geqslant1\}$ for any such $\{q_{n}\}$.
\end{rem}

\begin{lem}
\label{Finfound}If $F$ is a finite foundation for the subset $Q$
of the poset $P$, then $F\cap P_{\min}$ is also a finite foundation
for $Q$, so we may assume that $F\subseteq P_{\min}$. If $Q$ is
a lower set with a finite foundation, then we can take $F$ to be
$Q\cap P_{\min}$.
\end{lem}

\begin{proof}
If $q\in Q$ and $r\leqslant q$, find a minimal $p\in F$ with $p\leqslant r$
(as $F$ is finite); then $p$ is minimal in $P$, as otherwise we
could find $s\in F$ with $s\lneqq p$, so $p\in F\cap P_{\min}$.
Hence $F\cap P_{\min}$ is a finite foundation for $Q$. If $Q$ is
also a lower set then $F\subseteq Q$, and if $p\in Q\cap P_{\min}$
then we must have $p\in F$, so $F\cap P_{\min}=Q\cap P_{\min}$. 
\end{proof}

\section{\label{sec:Definitions-and-Basic}$P$-partitions: definitions, examples
and basic properties}

\subsection{Definitions}

The following definitions apply to a general topological space.
\begin{defn}
Let $W$ be a topological space, $(P,<)$ a poset or PO system, and
$\mathscr{X}=\{X_{p}\mid p\in P\}^{*}$ a partition of $W$: that
is, there is a subset $X$ of $W$ and a partition $\{X_{p}\mid p\in P\}$
of $X$ with each $X_{p}\neq\emptyset$. 

We will say that $\mathscr{X}$ is a \emph{$P$-partition }of $W$
if for all $p\in P$:

$\begin{cases}
\overline{X_{p}}\cap X=\bigcup_{q\leqslant p}X_{q} & \text{if }P\text{ is a poset;}\\
X_{p}^{\prime}\cap X=\bigcup_{q<p}X_{q} & \text{if }P\text{ is a PO system}.
\end{cases}$

For a subset $Y$ of $W$, we define its \emph{type }$T(Y)=\{p\in P\mid Y\cap X_{p}\neq\emptyset\}$. 

An open subset $A$ of $W$ is \emph{$p$-trim,} and we write $t(A)=p$,
if $T(A)=\{q\in P\mid q\geqslant p\}$, with additionally $|A\cap X_{p}|=1$
if $P$ is a PO system and $p\in P^{d}$. An open set is \emph{trim
}if it is $p$-trim for some $p\in P$. If $W$ is a Stone space,
we require trim sets to be compact and open (i.e.\ to be elements
of the associated Boolean ring). 

We write $\Trim(\mathscr{X})$ for the set of trim subsets of $W$,
$\Trim_{p}(\mathscr{X})$ for $\{A\in\Trim(\mathscr{X})\mid t(A)=p\}$
and $\widehat{P}$ for $\{t(A)\mid A\in\Trim(\mathscr{X})\}$. 

For $w\in W$, let $V_{w}=\{A\subseteq W\mid w\in A\wedge A\in\Trim(\mathscr{X})\}$
and let $I_{w}=\{t(A)\mid A\in V_{w}\}$, being the trim neighbourhoods
of $w$ and their types.

We say that $w\in W$ is a \emph{clean point }if $w\in X_{p}$ for
some $p$ and has a $p$-trim neighbourhood, and is otherwise a \emph{limit
point}.

A $P$-partition $\mathscr{X}$ of $W$ is a \emph{semi-trim }$P$-partition
if it also satisfies:
\begin{description}
\item [{ST1}] Every element of $W$ has a neighbourhood base of trim sets
(this implies that $X$ is dense in $W$);
\item [{ST2}] The partition is \emph{full}: that is, for each $p\in\widehat{P}$,
every element of $W$ with a neighbourhood base of $p$-trim sets
is an element of $X_{p}$;
\item [{ST3}] Clean points are dense in $X_{p}$ for $p\in\widehat{P}$. 
\end{description}
The semi-trim partition $\mathscr{X}$ is \emph{strongly semi-trim
}if all points in $X_{p}$ are clean for $p\in\widehat{P}$; is \emph{trim
}if all points in $X$ are clean (which requires $\widehat{P}=P$);
and is \emph{complete }if $X=W$.

If $\{Y_{p}\mid p\in P\}^{*}$ is a $P$-partition of the topological
space $Z$, we say that $\alpha:W\rightarrow Z$ is a \emph{$P$-homeomorphism}
if it is a homeomorphism such that $X_{p}\alpha=Y_{p}$ for all $p\in P$.
\end{defn}

\begin{rem}
$X_{p}$ contains clean points iff $p\in\widehat{P}$. For a semi-trim
partition, ST3 is equivalent to saying that if $A$ is open and $p\in T(A)\cap\widehat{P}$,
then $A$ contains a $p$-trim subset. 

If $P$ is a PO system and $\{X_{p}\mid p\in P\}^{*}$ is a $P$-partition
of $W$, then each $X_{p}$ is either discrete ($p\in P^{d}$) or
has no isolated points ($p\notin P^{d}$), and $\overline{X_{p}}\cap X=\bigcup_{q\leqslant p}X_{q}$
for all $p\in P$, as $\overline{X_{p}}=X_{p}\cup X_{p}^{\prime}$.

If $\{X_{p}\mid p\in P\}^{*}$ is a $P$-partition satisfying ST1
and ST3, then we can extend each $X_{p}$ for $p\in\widehat{P}$ by
adding to it all other elements with a neighbourhood base of $p$-trim
sets, to obtain a partition that also satisfies ST2, and it is easy
to see that this preserves the type of all the open sets. So condition
ST2 is a natural inclusion in the definition.

This section and Sections~\ref{sec:Completion}~and~\ref{sec:Closure-algebras-and}
concern $P$-partitions where $P$ may be either a PO system or a
poset. In Sections~\ref{sec:Characterising-partitions-by},~\ref{sec:Existence-Conditions}~and~\ref{sec:Uniqueness-Conditions}
however, $P$ will primarily be a PO system.
\end{rem}

\subsection{Examples}
\begin{example}
Let $X_{0}=C$ be a closed subset of the topological space $W$ such
that $X_{1}=W-C$ is dense in $W$, and let $P=\{0,1\}$. Then viewing
$P$ as a poset with $0\lneqq1$, $\mathscr{X}=\{X_{0},X_{1}\}$ is
a complete trim $P$-partition of $W$. An open subset of $W$ is
$0$-trim if it meets $C$, and otherwise is $1$-trim.

Suppose $W-C$ has no isolated points. If $C$ (as a space in its
own right) has no isolated points, then $\mathscr{X}$ is a $P$-partition
where $P$ is the PO system $\{0,1\}$ with $P^{d}=\emptyset$; while
if $C$ is discrete, then $\mathscr{X}$ is a $P$-partition where
$P$ is the PO system with $P^{d}=\{0\}$, with the $0$-trim sets
being the open sets containing just one element of $C$. If $C$ has
a mixture of isolated and non-isolated points, then $\mathscr{X}$
won't be a $P$-partition for any PO system $P$, but there may be
such a partition that refines $\mathscr{X}$.
\begin{example}
If $P=\{a,b\}$ with $a$ and $b$ unrelated, then trim $P$-partitions
of $W$ have the form $\{X_{a},X_{b}\}$ where $\{X_{a},X_{b}\}$
is a clopen partition of $W$. An open subset is $p$-trim ($p=a,b$)
iff it is a subset of $X_{p}$.
\begin{example}
More generally, if $C$ is any closed subset of a topological space
$W$, then we shall see in Section~\ref{sec:Closure-algebras-and}
that there is a complete trim partition of $W$ that refines the closure
subalgebra of $W$ generated by $C$.
\begin{example}
For any finite PO system $P$, Pierce~\cite{Pierce} showed how to
construct a compact Stone space $W$ with a complete $P$-partition,
and any such partition will necessarily be trim. The space $W$ is
of finite type, having only finitely many subsets invariant under
homeomorphisms of $W$.
\begin{example}
\label{exa:Primitive spaces}Spaces of finite type are a special case
of primitive Boolean algebras. Let $R$ be a countable primitive Boolean
algebra with Stone space~$W$. We shall see in Section~\ref{sec:Characterising-partitions-by}
that there is a canonical trim partition~$\mathscr{X}$ of~$W$
such that the trim elements of~$R$ are precisely its p.i.\ elements,
and that a countable Boolean ring is primitive iff its Stone space
admits a trim $P$-partition for some PO system $P$. However, this
partition will not in general be complete. 

For this canonical partition $\mathscr{X}$, our definition of clean
point corresponds to a \emph{homogeneous point }in Pierce~\cite[\S3.2]{PierceMonk},
trim sets correspond to \emph{uniform neighbourhoods}, and for $w\in W$,
our definition of $I_{w}$ corresponds to the \emph{characteristic
set }of $w$ defined in~\cite{Hansoul}.
\begin{example}
\label{Good example}However, not all $P$-partitions are even semi-trim.
Let $P$ be the poset $\{p_{1},p_{2},\ldots,q\}$, where each $p_{n}$
is a minimal element of $P$ and $p_{n}\lneqq q$ for all $n$. Let
$W$ be the (compact) Cantor set and let $E=\{x_{n}\mid n\geqslant1\}$
be an infinite discrete subset of $W$ (i.e.\ $E\cap E^{\prime}=\emptyset)$.
Let $X_{p_{n}}=\{x_{n}\}$ and $X_{q}=W-E$. Then $\{X_{p}\mid p\in P\}$
is a $P$-partition of $W$, as each $X_{p_{n}}$ is closed and $\overline{X_{q}}=W$.
If now $y$ is any accumulation point of $E$, which exist as $W$
is compact, then $y\in X_{q}$ but $y$ has no $q$-trim neighbourhoods.
Taking $Q=\{p_{n}\mid n\geqslant1\}$, we see that $Q$ is a lower
subset of $P$ but $X_{Q}$ is not closed.
\begin{example}
\label{exa:Ideal completion}Let $P=\mathbb{N}\cup\{\infty\}$, viewed
as a poset with the natural order. We shall see (Theorem~\ref{existence construction})
that we can find a Stone space $W$ with a trim $P$-partition $\mathscr{X}=\{X_{p}\mid p\in P\}^{*}$,
so that $X_{0}\subseteq\overline{X_{1}}\subseteq\overline{X_{2}}\subseteq\cdots\subseteq\overline{X_{\infty}}$,
with $X_{0}$ closed and $X_{\infty}$ open. Any compact open set
meeting $X_{m}$ will also meet $X_{n}$ for $m\leqslant n\leqslant\infty$,
and there will be such sets that miss $X_{n}$ for $n\lneqq m$ and
so are $m$-trim. Points in $W-X_{P}$ are limit points that have
$n$-trim neighbourhoods for each $n\in\mathbb{N}$ but are not elements
of $X_{\infty}$. There is no Stone space with a complete trim $P$-partition,
but there are two natural ways to extend $\mathscr{X}$ to a complete
semi-trim partition of $W$:
\end{example}

\end{example}

(a) expand $X_{\infty}$ to include all limit points. The expanded
partition will be a complete semi-trim $P$-partition but will not
be strongly semi-trim, as the expanded $X_{\infty}$ will contain
a mix of clean and limit points;

(b) let $Q=P\cup\{r\}$, where $n\lneqq r\lneqq\infty$ for all $n$,
let $Y_{p}=X_{p}$ for $p\in P$ and let $Y_{r}=W-X_{P}$ (i.e.\ all
the limit points). The resulting partition will be a complete strongly
semi-trim $Q$-partition of $W$, with $\widehat{Q}=P$, as all points
in $Y_{p}$ are clean for $p\in P$, but there are no $r$-trim sets
as every neighbourhood of a limit point meets $Y_{n}$ for some $n\in\mathbb{N}$.
\begin{example}
For a transitive relation $\prec$ on $P$, Ziegler~\cite{FlumZiegler}
defines a \emph{good partition }of a $T_{3}$-topological space $W$
to be a partition $\{X_{p}\mid p\in P\}$ such that $X_{p}^{\prime}=\bigcup_{q\prec p}X_{q}$
and $X_{q}\cap\{\bigcup_{q\nprec p}X_{p}\}^{\prime}=\emptyset$. If
$\prec$ is also antisymmetric, a good partition is precisely a complete
trim partition of the space.
\begin{example}
Myers~\cite[\S9]{Myers} defines a \emph{rank diagram }of a
Stone space and shows that an $\omega$-Stone space has a rank diagram
iff it is primitive. The rank diagram corresponds to our definition
of a complete strongly semi-trim partition of the underlying Stone
space.
\end{example}

\end{example}

\end{example}

\end{example}

\end{example}

\end{example}

\end{example}

\subsection{Basic properties of $P$-partitions}
\begin{prop}
\label{basic properties}($P$-partitions) Let $P$ be a poset or
PO system and let $\mathscr{X}=\{X_{p}\mid p\in P\}^{*}$ be a $P$-partition
of the topological space $W$.

(i) If $A$ is open in $W$, then $T(A)$ is an upper subset of $P$;

(ii) If $x\in X_{p}$ $(p\in P)$, then $x$ has a neighbourhood base
of $p$-trim sets iff $x$ has a $p$-trim neighbourhood (i.e.\ iff
$x$ is a clean point);

(iii) If $P$ satisfies the ACC and contains no infinite antichains
(e.g.\ if $P$ is finite), then a complete $P$-partition of $W$
is also a trim $P$-partition;

(iv) $X_{q}$ is infinite unless $q\in Q$, where $Q=P_{\min}$ if
$P$ is a poset and $Q=P_{\min}\cap P^{d}$ if $P$ is a PO system.
\end{prop}

\begin{proof}
(i) If $p\in T(A)$ and $q\geqslant p$, choose $x\in X_{p}\cap A$.
Then $x\in\overline{X_{q}}$, so $X_{q}\cap A\neq\emptyset$ and $q\in T(A).$

(ii) If $x\in X_{p}$ has a $p$-trim neighbourhood $A$, say, and
$B\subseteq A$ is an open neighbourhood of $x$, then $T(B)=T(A)$
by (a), as $p\in T(B)$, and if $|A\cap X_{p}|=1$ then $|B\cap X_{p}|=1$;
hence $B$ is a $p$-trim neighbourhood of $x$.

For (iii), let $\{X_{p}\mid p\in P\}$ be a complete $P$-partition
of $W$, and choose any $p\in P$. Let $Q=\{q\in P\mid q\ngeqslant p\}$.
Find a maximal $q_{1}\in Q$ and let $Q_{1}=\{q\in Q\mid q\leqslant q_{1}\}$,
using the ACC\@. If $Q_{1}\neq Q$, find a maximal $q_{2}\in Q-Q_{1}$
and let $Q_{2}=Q_{1}\cup\{q\in Q\mid q\leqslant q_{2}\}$. Repeating
this process, we will eventually find an antichain $\{q_{1},\ldots,q_{n}\}\subseteq Q$
such that for each $q\in Q$, we have $q\leqslant q_{j}$ for some
$j$, as $P$ contains no infinite antichains. Hence $\bigcup_{j\leqslant n}\overline{X_{q_{j}}}=\bigcup_{q\in Q}\overline{X_{q}}$. 

Now let $x$ be any element of $X_{p}$ and let $A$ be an open set
containing $x$, such that if $P$ is a PO system and $p\in P^{d}$
then $A\cap X_{p}=\{x\}$. By (ii) and completeness, it is enough
to find a $p$-trim neighbourhood of $x$. For each $j\leqslant n$
we have $x\notin\overline{X_{q_{j}}}$, so we can find an open set
$B_{j}$ with $x\in B_{j}\subseteq A$ and $B_{j}\cap\overline{X_{q_{j}}}=\emptyset$.
Then $B=\bigcap_{j\leqslant n}B_{j}$ is an open neighbourhood of
$x$ contained in $A$ which is disjoint from $\bigcup_{q\in Q}\overline{X_{q}}$,
with $B\cap X_{p}=\{x\}$ if $P$ is a PO system and $p\in P^{d}$.
So $B$ is a $p$-trim neighbourhood of $x$, as required.

(iv) is immediate, as if $q\notin Q$ then $X_{q}^{\prime}\supseteq X_{p}$
for some $p\in P$.
\end{proof}
The statements in the next Proposition will be used repeatedly and
have been labelled for ease of reference.
\begin{prop}
\label{basic properties-2 trim}(Semi-trim partitions) Let $P$ be
a poset or PO system and $\mathscr{X}=\{X_{p}\mid p\in P\}^{*}$ a
semi-trim $P$-partition of the topological space $W$.
\begin{description}
\item [{STP1\label{STP1}}] If $w\in W$, then $I_{w}$ is an ideal of
$\widehat{P}$ and:
\begin{description}
\item [{STP1A\label{STP1A}}] If $w\in X_{p}$ and $V\subseteq V_{w}$
is a neighbourhood base of $w$ then $p=\sup_{P}I_{w}=\sup_{P}\{t(A)\mid A\in V\}$;
\item [{STP1B\label{STP1B}}] $w$ is a clean point iff $I_{w}$ is a principal
ideal of $\widehat{P}$;
\end{description}
\item [{STP2\label{STP2}}] For all $q\in P$, we can find $p\in\widehat{P}$
such that $p\leqslant q$; and for all such $p$, there is an ideal
$J$ of $\widehat{P}$ such that $p\in J$ and $q=\sup_{P}J$;
\item [{STP3\label{STP3}}] If $P$ has the ACC, then $\mathscr{X}$ is
a complete trim $P$-partition of $W$;
\item [{STP4\label{STP4}}] If $w\in X_{q}$ is a limit point and $W$
is a $T_{1}$-space, then $w$ is not isolated in $X_{q}$;
\item [{STP5\label{STP5}}] If $P$ is a PO system and $W$ is a $T_{1}$-space,
then all points in $X_{P^{d}}$ are clean and $P^{d}\subseteq\widehat{P}$;
\item [{STP6\label{STP6}}] If $W$ is second-countable, then $\widehat{P}$
is countable.
\end{description}
\end{prop}

\begin{proof}
(STP1) Choose $w\in W$ and $A,B\in V_{w}$. As the partition is semi-trim,
we can find $C\in V_{w}$ such that $C\subseteq A\cap B$; then $t(C)\geqslant t(A)$
and $t(C)\geqslant t(B)$, and $t(C)\in I_{w}$. In addition, if $D\in\Trim(\mathscr{X})$
and $t(D)\lneqq t(A)$, then $A\cup D$ is $t(D)$-trim and $A\cup D\in V_{w}$,
so $t(D)\in I_{w}$. Hence $I_{w}$ is an ideal of $\widehat{P}$. 

If now $w\in X_{p}$, then $t(A)\leqslant p$ for all $A\in V$. If
$q\geqslant t(A)$ for all $A\in V$ then $X_{q}\cap A\neq\emptyset$;
hence $w\in\overline{X_{q}}$ and $q\geqslant p$, and so $p=\sup_{P}\{t(A)\mid A\in V\}$,
and taking $V=V_{w}$ we have $p=\sup_{P}I_{w}$. 

If further $w\in X_{p}$ is a clean point, then $p\in I_{w}$ and
$I_{w}=\{q\in\widehat{P}\mid q\leqslant p\}$. Conversely, if $w\in W$
and $I_{w}=\{q\in\widehat{P}\mid q\leqslant p\}$, choose $A\in V_{w}$
such that $t(A)=p$; then $t(B)=p$ for $B\in V_{w}$ and $B\subseteq A$,
so $w$ has a neighbourhood base of $p$-trim sets and $w\in X_{p}$
by fullness; hence $w$ is a clean point.

(STP2) If $q\in P$, choose $w\in X_{q}$ and $B\in V_{w}$; then
$t(B)\leqslant q$ and $t(B)\in\widehat{P}$; and if $p\in\widehat{P}$
with $p\leqslant q$, choose $B\in\Trim_{p}(\mathscr{X})$, $w\in B\cap X_{q}$,
let $J=I_{w}$ and use STP1A.

(STP3) Choose any $w\in W$. As $I_{w}\neq\emptyset$ and $P$ has
the ACC, we can find a maximal element $q$, say, of $I_{w}$. But
$I_{w}$ is an ideal, so $I_{w}=\{p\in\widehat{P}\mid p\leqslant q\}$.
Hence by~STP1B $w$ is a clean point in $X_{q}$, and so $\mathscr{X}$
is a complete trim $P$-partition of $W$.

(STP4) Suppose $w\in X_{q}$ is a limit point, let $A$ be any element
of $V_{w}$ and let $p=t(A)\lneqq q$. Let \textbf{$x\in A\cap X_{p}$
}be a clean point, and let $C$ be a $p$-trim neighbourhood of $x$
such that $w\notin C$, as $W$ is a $T_{1}$-space. Then $\text{\ensuremath{C\cap X_{q}\neq\emptyset}}$,
and so $w$ is not isolated in $X_{q}$.

(STP5) follows immediately from STP4. 

(STP6) For each $p\in\widehat{P}$, find a $p$-trim set $A_{p}$
and choose $x_{p}\in A_{p}\cap X_{p}$. We can find $C_{p}$ in the
countable base of open sets of $W$ such that $x_{p}\in C_{p}\subseteq A_{p}$,
and $C_{p}$ will be $p$-trim by STP2. Hence $\widehat{P}$ is countable.
\end{proof}
\begin{rem}
If $P$ is a poset and $\mathscr{X}$ a complete partition of the
topological space $W$, it can be shown that $\mathscr{X}$ is a trim
$P$-partition of $W$ if and only if (i) the label function $\tau:W\rightarrow P:w\mapsto p$
for $w\in X_{p}$ is continuous, and (ii) $\tau(A)$ is open whenever
$A$ is open, where $P$ is given the usual topology whereby all upper
sets are open. In this case it then follows that $X_{Q}$ is open
iff $Q$ is an upper subset of $P$ and $X_{Q}$ is closed iff $Q$
is a lower subset of $P$. 

For $P$-partitions of $W$, if $X_{Q}$ is closed, then $Q$ is a
lower subset of $P$ and if $X_{Q}$ is open, then $Q$ is an upper
subset of $P$, but the converse is not in general true --- e.g.\ see
Example~\ref{Good example}. However, for a complete $P$-partition,
if $Q$ is a lower set and has a finite ceiling, then $X_{Q}$ is
closed. 

For semi-trim $P$-partitions of an $\omega$-Stone space, it can
be shown that $\overline{X_{Q}}=\{w\in W\mid I_{w}\subseteq Q_{\downarrow}\cap\widehat{P}\}$.
In particular:

(i) if $\mathscr{X}$ is a trim $P$-partition of $W$, then $\overline{X_{Q}}\cap X_{P}=X_{Q_{\downarrow}}$,
and $X_{Q}$ is closed iff $Q$ is a lower subset of $P$ such that
$Q\cap\widehat{P}$ satisfies the ACC;

(ii) if $\mathscr{X}$ is a complete semi-trim $P$-partition of $W$,
$X_{Q}$ is closed iff $Q$ is a lower subset of $P$ and contains
the supremum of every ideal of $Q\cap\widehat{P}$.
\end{rem}

Our next results look at how the topological properties of isolated
points and compactness reflect the order properties of the poset /
PO system.
\begin{prop}
\label{Isolated propn}(Isolated points)

Let $P$ be a poset or PO system, $W$ a $T_{1}$ topological space
and $\{X_{p}\mid p\in P\}^{*}$ a semi-trim $P$-partition of $W$.
Then the isolated points of $W$ are precisely the isolated points
of each $X_{p}$ for which $p\in P_{\max}$, which if $P$ is a PO
system is just $X_{Q}$ where $Q=P^{d}\cap P_{\max}$.
\end{prop}

\begin{proof}
Suppose $w\in W$ is isolated. Then $\{w\}$ must be a trim set: say
$\{w\}$ is $p$-trim. Then clearly $p\in P_{\max}$, $w\in X_{p}$
and $w$ is isolated in $X_{p}$. 

Conversely, if $p\in P_{\max}$ and $w$ is isolated in $X_{p}$,
then $w$ is a clean point in $X_{p}$ (STP4), so we can find a $p$-trim
$A$ such that $A\cap X_{p}=\{w\}$. If now $y\in A$ with $y\neq w$,
then $y$ has a $p$-trim neighbourhood contained in $A-\{w\}$, as
$p\in P_{\max}$, which is impossible. Hence $A=\{w\}$, so $\{w\}$
is open and $w$ is isolated, as required. The final statement is
immediate.
\end{proof}
\begin{prop}
\label{Compact means FF}(Compactness)

Let $P$ be a poset or PO system and $W$ a topological space that
admits a semi-trim $P$-partition $\{X_{p}\mid p\in P\}^{*}$. Then:

(i) if $A\subseteq W$ is compact and either $A$ is open or $T(A)$
is a lower set, then $T(A)$ viewed as a poset in its own right has
a finite foundation;

(ii) if $W$ is compact, then $P$ has a finite foundation;

(iii) if $\overline{X_{p}}$ is compact, then $p\in P_{\Delta}$.
\end{prop}

\begin{proof}
For (i), let $A$ be a compact subset of $W$, and for each $x\in A$
find $A_{x}\in V_{x}$, with $A_{x}\subseteq A$ if $A$ is open.
By compactness, we can find a finite subset $\{A_{x_{1}},\ldots,A_{x_{n}}\}$
of trim sets that cover $A$. Now take $F=\{t(A_{x_{i}})\mid i\leqslant n\}\cap T(A)$:
this is a finite foundation for $T(A)$, as for any $q\in T(A)$ and
$y\in A\cap X_{q}$ we have $q\geqslant t(A_{x_{i}})$ if $y\in A_{x_{i}}$,
and $t(A_{x_{i}})\in T(A)$ as either $A_{x_{i}}\subseteq A$ or $T(A)$
is a lower set. 

(ii) follows immediately from (i), taking $A=W$.

For (iii) $T(\overline{X_{p}})=\{q\in P\mid q\leqslant p\}$, so if
$\overline{X_{p}}$ is compact then by (i) $T(\overline{X_{p}})$
has a finite foundation and $p\in P_{\Delta}$.
\end{proof}
\begin{rem}
Example~\ref{Good example} shows that a compact $W$ may admit a
(non-semi-trim) $P$-partition even though $P$ does not have a finite
foundation.
\end{rem}

\begin{prop}
Let $P$ be a poset or PO system, $W$ a topological space that admits
a $P$-partition that is either complete or semi-trim, and $A$ a
compact closed subset of $W$. Then every descending chain in $T(A)$
has a lower bound.
\end{prop}

\begin{proof}
For semi-trim partitions, this follows easily from Proposition~\ref{Compact means FF}(i),
as $T(A)$ is a lower set. Suppose$\mathscr{X}=\{X_{p}\mid p\in P\}^{*}$
is a complete $P$-partition of $W$, $A$ is a compact closed subset
of $W$, and $p_{1}\geqslant p_{2}\geqslant\cdots$ is a descending
chain of elements in $T(A)$. Then $Y=\bigcap_{i\geqslant1}(A\cap\overline{X_{p_{i}}})\neq\emptyset$,
as $A$ is compact and closed. Choose $y\in Y$: say $y\in X_{q}$;
then $q$ is a lower bound for $\{p_{i}\}$.
\end{proof}
\begin{example}
The lower bound above may not be unique, even for a complete trim
partition. For example, consider $P=\{p_{1},p_{2},p_{3},\ldots,q,r\}$,
where $p_{1}\gneqq p_{2}\gneqq p_{3}\gneqq\cdots$, with $p_{n}\gneqq q$
and $p_{n}\gneqq r$ for all $n$. $P$ satisfies the ACC, and by
Theorem~\ref{Existence theorem} below there is a Stone space $W$
with a complete trim $P$-partition, but the descending chain $\{p_{n}\}$
has two lower bounds.
\end{example}

\section{\label{sec:Characterising-partitions-by}Trim partitions and primitive
Boolean rings}

We have defined semi-trim partitions by reference to a topological
space. For trim partitions of the Stone space of a Boolean ring, we
can achieve an equivalent definition via the Boolean ring and a ``structure
function'' as introduced by Hanf~\cite{Hanf}, without reference
to the Stone space. This will provide a convenient way of constructing
trim partitions by first defining an appropriate structure function.

Trim sets are ``topologically pseudo-indecomposable'' in that if
$A$ is $p$-trim and is the disjoint union of two elements of $R$,
then one of those sets will also be $p$-trim. Our first goal is to
show that when $P$ is a PO system, the ideal of $R$ generated by
a trim set is algebraically pseudo-indecomposable, using the following
version of Vaught's Theorem, as cited in~\cite[1.1.3]{PierceMonk}. 
\begin{thm}
\label{(Vaught)}(Vaught) Suppose $\sim$ is a relation between elements
of countable Boolean algebras $R$ and $S$ such that:

(i) $1_{R}\sim1_{S}$;

(ii) if $C\sim0_{S}$, then $C=0_{R}$; and vice versa;

(iii) if $C\sim(D_{1}\dotplus D_{2})$ ($D_{1},D_{2}\in S$), then
we can write $C=C_{1}\dotplus C_{2}$, with $C_{i}\in R$ and $C_{i}\sim D_{i}$
$(i=1,2)$; and vice versa.

Then there is an isomorphism $\alpha:R\rightarrow S$ such that each
$C\in R$ can be expressed as $C=C_{1}\dotplus\ldots\dotplus C_{n}$
where $C_{i}\sim C_{i}\alpha$ for all $i\leqslant n$.
\end{thm}

The following Proposition regarding semi-trim partitions of Stone
spaces is fundamental. This, together with Theorem~\ref{Thm:same mu=00003Dhomeom}
and Corollary~\ref{cor:trim=00003Dpi}, generalise some known results
for Boolean algebras: e.g.\ see Myers~\cite[\S9,14,15]{Myers}.
\begin{prop}
\label{Propn std split of trim set}Let $P$ be a poset or PO system,
and $R$ a Boolean ring whose Stone space admits a semi-trim $P$-partition
$\{X_{p}\mid p\in P\}^{*}$. 

Let $A\in R$. Then there is a finite partition of $A$ containing
precisely $n_{p}$ $p$-trim sets for each $p\in F$, writing $F=T(A)_{\min}$,
where:

(i) if $P$ is a PO system, then $n_{p}=1$ for $p\in F-P^{d}$ and
$n_{p}=|A\cap X_{p}|$ for $p\in F\cap P^{d}$;

(ii) if $P$ is a poset, then $n_{p}=1$ for each $p\in F$.
\end{prop}

\begin{proof}
By Proposition~\ref{Compact means FF} and Lemma~\ref{Finfound},
$F=T(A)_{\min}$ is a finite foundation for $T(A)$ as $A$ is compact
and open.

For each $x\in A$ find a trim set $A_{x}\subseteq A$ containing
$x$. By compactness, we can find a finite subset $\{A_{x_{1}},\ldots,A_{x_{n}}\}$
of trim sets that cover $A$. Let $\mathscr{B}=\{B_{i}\}$ be the
atoms of the finite Boolean ring generated by $\{A_{x_{j}}\}$, so
that $\mathscr{B}$ is a finite partition of $A$ with each $B_{i}$
a subset of one of the $A_{x_{j}}$. For each $p\in F$, we can find
a $B_{i}$ meeting $X_{p}$, and any $B_{i}$ meeting $X_{p}$ must
be $p$-trim, as it is a subset of a $q$-trim subset of $A$ for
some $q\in T(A)$, and $T(B_{i})$ contains the minimal element $p$
of $T(A)$, hence $q=p$. 

For (i), we can now write $\mathscr{B}=\mathscr{B}_{1}\cup\mathscr{B}_{2}\cup\mathscr{B}_{3}$,
where $\mathscr{B}_{1}$ contains all the elements of $\mathscr{B}$
that meet $X_{p}$ for some $p\in F\cap P^{d}$ (which are each $p$-trim),
$\mathscr{B}_{2}$ contains a single $p$-trim element of $\mathscr{B}$
for each $p\in F-P^{d}$, and $\mathscr{B}_{3}$ contains the remaining
elements of $\mathscr{B}$. For each $C\in\mathscr{B}_{3}$, we have
$C\subseteq A_{x_{j}}$ for some $j$, and we can find $r\in F$ with
$r\leqslant t(A_{x_{j}})$. Find $D\in\mathscr{B}_{1}\cup\mathscr{B}_{2}$
such that $t(D)=r$; then $D\cup C$ is still $r$-trim, as $T(C)\subseteq T(D)$,
and if $r\in F\cap P^{d}$ then $C\cap X_{r}=\emptyset$ by definition
of $\mathscr{B}_{1}$; so we can replace $D$ with $D\cup C$. Proceeding
in this way, we obtain the desired finite partition $\widetilde{\mathscr{B}_{1}}\cup\widetilde{\mathscr{B}_{2}}$
of $A$ as, by counting, $\widetilde{\mathscr{B}_{1}}$ will contain
$|A\cap X_{p}|$ $p$-trim sets for each $p\in F\cap P^{d}$.

For (ii), write $\mathscr{B}=\mathscr{B}_{2}\cup\mathscr{B}_{3}$,
where $\mathscr{B}_{2}$ contains one $p$-trim element of $\mathscr{B}$
for each $p\in F$ and $\mathscr{B}_{3}=\mathscr{B}-\mathscr{B}_{2}$,
and proceed as in (i) to create a partition $\widetilde{\mathscr{B}_{2}}$
of $A$ with exactly one $p$-trim set for each $p\in F$.
\end{proof}
\begin{defn}
If $P$ is a poset or PO system, $R$ is a Boolean ring whose Stone
space admits a semi-trim $P$-partition $\mathscr{X}=\{X_{p}\mid p\in P\}^{*}$,
and $A\in R$, we define the \emph{$\mathscr{X}$-measure of $A$}
as the formal sum $\mu(A)=\sum_{p\in F}n_{p}.p$, where $F=T(A)_{\min}$
and $\{n_{p}\mid p\in F\}$ are as in Proposition~\ref{Propn std split of trim set},
setting $\mu(0)=0$. 
\end{defn}

\begin{rem}
For semi-trim partitions of Stone spaces, this Proposition provides
a more precise way of describing the type of elements of the associated
Boolean ring, as it is easy to see that any minimal decomposition
of $A$ into trim sets has the stated form. The formal sum $\mu(A)$
can be viewed as an element of a monoid: e.g.\ see Myers~\cite[\S9]{Myers}. 

If $P$ is a poset, we could equivalently define $\mu(A)$ for $A\in R$
to be the finite set of minimal elements of $T(A)$. This will be
an antichain, and every finite antichain in $P$ can arise in this
way. The trim sets are just those for which $\mu(A)$ is a singleton
set. 
\end{rem}

\begin{thm}
\label{Thm:same mu=00003Dhomeom}Let $P$ be a PO system and $\mathscr{X}$
and $\mathscr{Y}$ trim $P$-partitions of the Stone spaces $X$ and
$Y$ of countable Boolean rings $R$ and $S$ respectively. Suppose
$A\in R$ and $B\in S$ such that $\mu(A)=\mu(B)$, and that $a\in A$
and $b\in B$ such that $I_{a}=I_{b}$. 

Then there is a $P$-homeomorphism $\alpha:A\rightarrow B$ such that
$a\alpha=b$.
\end{thm}

\begin{proof}
Define a relation between the Boolean algebras $(A)$ and $(B)$ by
setting $C\sim D$ iff $\mu(C)=\mu(D)$, and either $a\in C$ and
$b\in D$ or $a\notin C$ and $b\notin D$, for $C\subseteq A$ and
$D\subseteq B$. Clearly $C\sim0$ iff $C=0$, and $A\sim B$.

We must show that $\sim$ satisfies criterion (iii) of Theorem~\ref{(Vaught)}.
By symmetry, it is enough to consider the case where $C\subseteq A$,
$C\sim D$ and $D=D_{1}\dotplus D_{2}$, where $D_{1},D_{2}\subseteq B$.
Let $T:R\rightarrow2^{P}$ be the type function of $\mathscr{X}$;
write $G=T(C)_{\min}$, and if $b$ is a clean point let $q=\sup_{P}I_{b}\in\widehat{P}$,
so that $b\in X_{q}$.

Using Proposition~\ref{Propn std split of trim set}(i), write $D_{1}=E_{1}\dotplus\ldots\dotplus E_{m}$
and $D_{2}=E_{m+1}\dotplus\ldots\dotplus E_{n}$, where each $E_{j}$
is $p_{j}$-trim for some $p_{j}$; if $b\in D$, we may assume that
$b\in E_{1}$, and if also $b$ is a clean point then choose $E_{1}\subseteq D_{1}$
to be a $q$-trim neighbourhood of $b$ (applying Proposition~\ref{Propn std split of trim set}(i)
to $D_{1}-E_{1}$). For each $j$, we can find distinct $y_{j}\in E_{j}\cap Y_{p_{j}}$
and $x_{j}\in C\cap X_{p_{j}}$, as $C\cap X_{p_{j}}$ is infinite
unless $p_{j}\in G\cap P^{d}$ (Proposition~\ref{basic properties}(iv)),
in which case $|C\cap X_{p_{j}}|=|D\cap Y_{p_{j}}|$ as $\mu(C)=\mu(D)$.
Find disjoint $p_{j}$-trim neighbourhoods $F_{j}\subseteq C$ of
each $x_{j}$ and let $F=F_{1}\dotplus\ldots\dotplus F_{n}$. 

If $C-F\neq\emptyset$, write $C-F=F_{n+1}\dotplus\ldots\dotplus F_{n+k}$,
where each $F_{n+i}$ is trim. For each $i\leqslant k$, we can find
$r\in G$ such that $r\leqslant t(F_{n+i})$, and $j\leqslant n$
such that $p_{j}=r$, as one of the $E_{j}$ must be $r$-trim. Then
$F_{j}\cup F_{n+i}$ will still be $r$-trim, as $T(F_{n+i})\subseteq T(F_{j})$,
and if $r\in P^{d}$ then $F_{n+i}\cap X_{r}=\emptyset$ as $|C\cap X_{r}|=|D\cap Y_{r}|=|F\cap X_{r}|<\infty$.
So we can replace $F_{j}$ with $F_{j}\cup F_{n+i}$. Proceeding in
this way, we may assume that $C=F_{1}\dotplus\ldots\dotplus F_{n}$,
with $F_{j}$ being $p_{j}$-trim.

If $a\notin C$ and $b\notin D$, or if $a\in F_{1}$ and $b\in E_{1}$,
we can take $C_{1}=F_{1}\dotplus\ldots\dotplus F_{m}$ and $C_{2}=F_{m+1}\dotplus\ldots\dotplus F_{n}$
to finish, as $\mu(C_{i})=\mu(D_{i})$ ($i=1,2$). Otherwise we have
$b\in E_{1}$ and $a\in F_{j}$, say, with $j>1$. There are 3 cases.

Case 1: if $b$ is a clean point and $E_{1}$ and $F_{j}$ are both
$q$-trim, we can swap $F_{1}$ and $F_{j}$ to obtain $a\in F_{1}$.

Case 2: if $b$ is a clean point, $E_{1}$ is $q$-trim but $F_{j}$
is $r$-trim where $r\lneqq q$, then $a$ is clean (as $I_{a}=I_{b}=\{r\in P\mid r\leqslant q\}$),
so we can find a $q$-trim neighbourhood $F_{j1}\subseteq F_{j}$
of $a$; let $F_{j2}=F_{j}-F_{j1}$ which is $r$-trim. Now replace
$F_{1}$ with $F_{j1}$, and $F_{j}$ with the $r$-trim set $F_{1}\cup F_{j2}$,
to obtain $a\in F_{1}$.

Case 3: if $b$ is a limit point, then $I_{b}$ is an ideal that does
not contain its own supremum, so we can find $F_{j1}\subseteq F_{j}$
such that $F_{j1}$ is an $r$-trim neighbourhood of $b$ for some
$r\in I_{b}$ such that $r\gneqq p_{1}$ and $r\gneqq p_{j}$, as
$\{p_{1},p_{j}\}\subseteq I_{b}$. Replace $F_{1}$ with $F_{1}\cup F_{j1}$
and $F_{j}$ with $F_{j}-F_{j1}$ to complete this step.

So by Theorem~\ref{(Vaught)} there is an isomorphism $\beta:(A)\rightarrow(B)$,
and each $C\subseteq A$ can be expressed as $C=C_{1}\dotplus\ldots\dotplus C_{n}$
where $C_{i}\sim C_{i}\beta$ for all $i\leqslant n$, so that $T(C_{i}\beta)=T(C_{i})$,
and $C_{i}$ is $p$-trim iff $\mu(C_{i})=1.p$ iff $C_{i}\beta$
is $p$-trim. But $T(C\cup D)=T(C)\cup T(D)$, and so $\beta$ is
type-preserving on all of $(A)$; moreover, if $C$ is $p$-trim for
$p\in P^{d}$, then only one of the $C_{i}$ is $p$-trim, so $|C\beta\cap Y_{p}|=1$
and $C\beta$ is $p$-trim. Hence $\Trim_{p}(\mathscr{X})\beta=\Trim_{p}(\mathscr{Y})$
for all $p\in P$.

Let $\alpha:A\rightarrow B$ be the Stone space homeomorphism induced
by $\beta$. If $x\in A\cap X_{p}$, then $x$ and hence also $x\alpha$
has a neighbourhood base of $p$-trim sets. It follows that $x\alpha\in B\cap Y_{p}$,
as $\mathscr{Y}$ is full, and so $(X_{p}\cap A)\alpha\subseteq Y_{p}\cap B$.
Equality follows by considering $\alpha^{-1}$, so $\alpha$ is a
$P$-homeomorphism. Finally for $C\subseteq A$, $b\in C\alpha$ iff
$a\in C$ iff $a\alpha\in C\alpha$, from which it follows easily
that $a\alpha=b$, as required.
\end{proof}
\begin{cor}
\label{cor:trim=00003Dpi}Let $P$ be a PO system and $\mathscr{X}$
a trim $P$-partition of the Stone space $W$ of the countable Boolean
ring $R$. Then 

(i) if $A\in R$ is trim, then $(A)$ is pseudo-indecomposable;

(ii) for $a,b\in W$, there is a $P$-homeomorphism $\alpha$ of $W$
such that $a\alpha=b$ iff $I_{a}=I_{b}$.
\end{cor}

\begin{proof}
(i) If $A=A_{1}\dotplus A_{2}$, with $A_{i}\in R$, then either $A_{1}$
or $A_{2}$ must be $t(A)$-trim, and so by Theorem~\ref{Thm:same mu=00003Dhomeom}
(choosing $a=b$) either $(A_{1})$ or $(A_{2})$ is isomorphic to
$(A)$, and hence $(A)$ is pseudo-indecomposable. 

(ii) $\Rightarrow$ is immediate as if $A\in V_{a}$ then $A\alpha\in V_{b}$
and $t(A\alpha)=t(A)$ for any $P$-homeomorphism $\alpha$ of $W$.

$\Leftarrow$ Find $A\in R$ such that $a,b\in A$ and apply Theorem~\ref{Thm:same mu=00003Dhomeom}
with $B=A$ to obtain a $P$-homeomorphism $\alpha$ of $A$ such
that $a\alpha=b$, which can be extended to a $P$-homeomorphism of
$W$ by setting $w\alpha=w$ for $w\in W-A$. 
\end{proof}
\begin{rem}
\label{rem:Index partition}If $\mathscr{X}$ is a trim $P$-partition
of $W$, then we can define a complete partition $\mathscr{Z}$ whose
elements are the orbits of $P$-homeomorphisms of $W$. By the previous
Corollary, elements of $\mathscr{Z}$ have the form $\{w\in W\mid I_{w}=J\}$
for ideals $J$ of $P$. This generalises the \emph{orbit diagram
}$\mathscr{Y}$ of a Stone space $W$ (see~\cite[\S15]{Myers}),
which consists of the orbits of elements of $W$ under homeomorphisms
of $W$. If $W$ is a primitive space, its orbit diagram is a complete
strongly semi-trim partition of $W$, whose clean points form the
trim partition $\mathscr{X}$ referred to in Example~\ref{exa:Primitive spaces}.
For this canonical trim partition, the converse of (i) is also true
(i.e.\ p.i.\ sets are trim), but this is not true for a general
trim $\mathscr{X}$.
\end{rem}

We now recall Hanf's definition of a structure function and show how
each such function is grounded in a trim partition of the Stone space.
\begin{defn}
\label{def:structure}(Hanf~\cite{Hanf}) If $P$ is a PO system,
$R$ is a Boolean ring and $S\subseteq R-\{0\}$, a surjective map
$u:S\rightarrow P$ \emph{structures} $R$ if (writing $S_{p}=\{A\in S\mid u(A)=p\}$):
\end{defn}

\begin{description}
\item [{S1}] $S$ disjointly generates $R$
\item [{S2}] if $A,B\in S$ and $B\subseteq A$ then $u(B)\geqslant u(A)$
\item [{S3}] if $A,B,C\in S_{p}$ and $A\supseteq B\dotplus C$, then $p<p$
\item [{S4}] if $A\in S_{p}$ and $A=B\dotplus C$, then we can find $D\in S_{p}$
such that $D\subseteq B$ or $D\subseteq C$
\item [{S5}] if $A\in S$, and $u(A)=p<q$, then we can find $B,C\in S$
such that $A\supseteq B\dotplus C$, $u(B)=p$ and $u(C)=q$
\end{description}
\begin{rem}
Hanf's original definition applied to Boolean algebras and required
$P$ to have a unique minimum element, being the ``type'' of $1_{R}$.
We have extended the definition to include Boolean rings and dropped
the unique minimum element requirement. 

There is also a choice to be made as to the direction of the order
relation. Hanf along with many other authors adopt the convention
for Boolean algebras whereby the unit element is of maximum type in
$P$. This is certainly a natural choice when considering (for example)
isomorphism classes of ideals $(A)$ for $A\in R$. However, when
working with partitions of the Stone space, it seems more intuitive
that partition elements $X_{p}$ for $p\in P^{d}$, which are discrete
and in some cases singleton subsets, correspond to \emph{minimum }elements
$p\in P$. We have adopted the latter convention, consistent with
the order direction in Ziegler's definition of a good partition (\cite{FlumZiegler}),
with the benefit that the open subsets of $W$ of the form $X_{Q}$
correspond to upper subsets $Q$ of $P$, which are the open subsets
under the usual Alexandroff topology on $P$.
\end{rem}

\begin{prop}
\label{structure=00003Dtrim}Let $P$ be a PO system and $R$ a countable
Boolean ring with Stone space $W$. 

(i) If $u:S\rightarrow P$ structures $R$, then there is a unique
trim $P$-partition $\mathscr{X}(u)$ of $W$ such that $S_{p}\subseteq\Trim_{p}(\mathscr{X}(u))$
for all $p\in P$;

(ii) if $\mathscr{X}$ is a trim $P$-partition of $W$, $t:\Trim(\mathscr{X})\rightarrow P$
structures $R$.
\end{prop}

\begin{proof}
(i) Suppose that $u:S\rightarrow P$ structures $R$. For $p\in P$
and $w\in W$, let $S_{p}(w)=\{A\in S_{p}\mid w\in A\}$, and define
a partition $\mathscr{X}(u)=\{X_{p}\mid p\in P\}^{*}$ of $W$, where
$X_{p}=\{w\in W\mid S_{p}(w)\text{ is a neighbourhood base for }w\}$.

We claim first that $A\cap X_{p}\neq\emptyset$ when $A\in S_{p}$.
Let $\{B_{1},B_{2},\ldots\}$ be an enumeration of $R$. Let $A_{0}=A$,
and using S4 choose $A_{n}\in S_{p}$ successively for $n\geqslant1$
such that $A_{n}\subseteq A_{n-1}$ and (writing $A_{n-1}=(A_{n-1}\cap B_{n})\dotplus(A_{n-1}-B_{n})$)
either $A_{n}\subseteq B_{n}$ or $A_{n}\cap B_{n}=\emptyset$. Then
$\bigcap_{n\geqslant1}A_{n}$ is non-empty by compactness and must
be a singleton $\{x_{A}\}$, say, as for any $x,y\in W$ we can find
$n$ such that $x\in B_{n}$ and $y\notin B_{n}$. It follows that
$x_{A}\in X_{p}$ and $p\in T(A)$.

Fix $x\in X_{p}$ and let $A\in S_{p}(x)$; we claim that $A$ is
$p$-trim. If $y\in A\cap X_{q}$ and $y\neq x$, find disjoint $B\in S_{p}$
and $C\in S_{q}$ such that $x\in B\subseteq A$ and $y\in C\subseteq A$.
Then $q>p$ by S2 and S3, and if $p\in P^{d}$ then $q\neq p$ and
so $|A\cap X_{p}|=1$. Conversely, if $q>p$, then by S5 we can find
$C\in S_{q}$ such that $C\subseteq A$, and so $q\in T(C)\subseteq T(A)$;
and if also $q=p$ then $|A\cap X_{p}|\geqslant2$. So $A$ is $p$-trim.
Letting $A$ run through the elements of $S_{p}(x)$, we see that
$x\in X_{q}^{\prime}$ iff $q>p$.

Hence $\mathscr{X}(u)$ is indeed a $P$-partition of $W$, every
element of $W$ has a neighbourhood base of trim sets (by S1) and
all points in $X_{p}$ are clean. Suppose $w\in W$ has a neighbourhood
base of $p$-trim sets and let $A$ be a $p$-trim set containing
$w$. Using S1, we can find $B\in S$ and $C\in R$ such that $w\in C\subseteq B\subseteq A$
and $C$ is $p$-trim. Hence $B$ is $p$-trim; but $B$ is $u(B)$-trim,
hence $u(B)=p$ and $B\in S_{p}(w)$. So $w\in X_{p}$ and $\mathscr{X}(u)$
is full, as required. 

Finally, suppose $\mathscr{Y}=\{Y_{p}\mid p\in P\}^{*}$ is a trim
$P$-partition of $W$ such that $S_{p}\subseteq\Trim_{p}(\mathscr{Y})$
for all $p\in P$. If $S_{p}(w)$ is a neighbourhood base for $w$
then $w\in Y_{p}$, as $S_{p}(w)\subseteq\Trim_{p}(\mathscr{Y})$
and $\mathscr{Y}$ is full; hence $X_{p}\subseteq Y_{p}$. Conversely,
if $w\in Y_{p}$, choose $B\in\Trim_{p}(\mathscr{Y})$ such that $w\in B$;
then if $w\in A\subseteq B$ and $A\in S$, then $A\in\Trim_{p}(\mathscr{Y})$,
so $A\in S_{p}(w)$; hence $w\in X_{p}$ and $\mathscr{Y}=\mathscr{X}$.

(ii) Suppose $\mathscr{X}$ is a trim $P$-partition of $W$. We must
show that $t:\Trim(\mathscr{X})\rightarrow P$ structures $R$. S1
follows from Proposition~\ref{Propn std split of trim set}(i), and
S2 is immediate. If $p\nless p$ and $A\in R$ is $p$-trim, then
$A\cap X_{p}$ is a singleton, from which S3 follows. For S4, we can
assume that $B\cap X_{p}\neq\emptyset$, so we can find a $p$-trim
$D\subseteq B$. For S5, if $q\neq p$, find $x\in A\cap X_{q}$ and
a $q$-trim neighbourhood $C\subseteq A$ of $x$, and we can take
$B=A-C$ to complete, since $C\cap X_{p}=\emptyset$ and so $B\cap X_{p}\neq\emptyset$.
If $q=p$, then $A\cap X_{p}$ has no isolated points, so we can find
$x,y\in A\cap X_{p}$ and a $p$-trim neighbourhood $B$ of $x$ such
that $y\notin B$, and letting $C=A-B$, we must have $t(B)=t(C)=p$,
as required.
\end{proof}
\begin{rem}
A simplified version of this proof shows that if $P$ is a poset,
$R$~is a Boolean ring, $S\subseteq R-\{0\}$ and $u:S\rightarrow P$
satisfies conditions S1, S2, S4 and a new condition S6 (``if $A\in S$,
and $u(A)=p\leqslant q$, then we can find $C\in S$ such that $C\subseteq A$
and $u(C)=q$''), then there is a unique trim $P$-partition $\mathscr{X}$
of the Stone space of $R$ such that $S_{p}\subseteq\Trim_{p}(\mathscr{X})$
($p\in P$).
\end{rem}

We can now state the main result in this section.
\begin{thm}
\label{Primitive=00003Dtrim}A countable Boolean ring is primitive
iff its Stone space admits a trim $P$-partition for some PO system
$P$.
\end{thm}

\begin{proof}
$\Leftarrow$ Follows directly from Corollary~\ref{cor:trim=00003Dpi}
and Proposition~\ref{Propn std split of trim set}(i).

$\Rightarrow$ Suppose $R$ is a countable primitive Boolean ring
with Stone space~$W$. For $A\in R$, write $[A]$ for the isomorphism
class of the ideal $(A)$. Let $S\subseteq R$ consist of all the
p.i.\ elements of $R$, and let $P=\{[A]\mid A\in S\}$, with the
relation $[A]<[B]$ iff $[A]\times[B]\cong[A]$. By Hanf~\cite[Theorem~5.4]{Hanf},
the map $u:S\rightarrow P:A\mapsto[A]$ satisfies Definition~\ref{def:structure}.
It follows from Proposition~\ref{structure=00003Dtrim}(i) that $\mathscr{X}(u)$
is a trim $P$-partition of $W$, in which all the p.i.\ sets are
trim.
\end{proof}
\begin{rem}
The proof of~\cite[Theorem~5.4]{Hanf} uses results about ordered
bases. As an alternative, the map $u:S\rightarrow P:A\mapsto[A]$
is easily seen to satisfy conditions S1, S2, S4 and S5, and the following
Lemma establishes S3.
\end{rem}

\begin{lem}
Suppose $R$ is a countable pseudo-indecomposable Boolean algebra
containing an element $E$ such that $(E)\cong R\times R$. Then $R\cong R\times R$.
\end{lem}

\begin{proof}
Let $S=R\times R$. We will term an element $F$ of $R$ or $S$ an
\emph{$R$-element }if $(F)\cong R$. For $A\in R$ and $B\in S$,
let $A\sim B$ if either $(A)\cong(B)$ or both $(A)$ and $(B)$
contain $R$-elements. Clearly $1_{R}\sim1_{S}$ (e.g.\ $(1_{R},0)\in S$
is an $R$-element) and if $A\sim0$ then $A=0$ and vice versa. Suppose
that $A\sim B$ and $B=B_{1}\dotplus B_{2}$; we must find $A_{i}$
$(i=1,2)$ such that $A=A_{1}\dotplus A_{2}$ and $A_{i}\sim B_{i}$
$(i=1,2)$.

If $(A)\cong(B)$, take $A_{1}$ and $A_{2}$ to be the isomorphic
images of $B_{1}$ and $B_{2}$ respectively. Suppose instead that
$C\subseteq A$ and $D\subseteq B$ with $(C)\cong(D)\cong R$. Then
$D=(D\cap B_{1})\dotplus(D\cap B_{2})$, so we may suppose that $(D\cap B_{1})\cong R$
as $(D)$ is p.i., and so $B_{1}$ contains an $R$-element. There
are 2 cases:

Case 1: $B_{2}$ has no $R$-elements. Then $B_{2}\in S\cong(E)\subseteq R\cong(C)$,
so we can find $A_{2}\subseteq C$ such that $(A_{2})\cong(B_{2})$.
Let $A_{1}=C-A_{2}$; then $C=A_{1}\dotplus A_{2}$ and so $A_{1}$
must be an $R$-element as $R$ is p.i.\ and $A_{2}$ has no $R$-elements.
Hence $A_{1}\sim B_{1}$ and $A_{2}\sim B_{2}$.

Case 2: $B_{2}$ contains an $R$-element. Then $R\times R\cong(E)\subseteq R\cong(C)$,
so (considering $(1_{R},0)$ and $(0,1_{R})\in R\times R$) we can
find $A_{1}\subseteq C$ such that $(A_{1})\cong R$ and such that
$C-A_{1}$, and so also $A_{2}=A-A_{1}$, contains an $R$-element.
Hence $A_{i}\sim B_{i}$ for $i=1,2$.

The case when $A\sim B$ and $A=A_{1}\dotplus A_{2}$ is identical,
except that for case 1 (when $A_{2}$ has no $R$-elements) we have
$A_{2}\in R$, so we can find $B_{2}\subseteq D$ such that $(A_{2})\cong(B_{2})$. 

Hence $\sim$ satisfies the Vaught criteria of Theorem~\ref{(Vaught)}
and $R\cong R\times R$. 
\end{proof}

\section{\label{sec:Completion}Extensions, clean interior and completion
of a semi-trim partition}

In this section we look first at conditions under which a $P$-partition
of an $\omega$-Stone space can be ``regularly'' extended to a $Q$-partition,
where $P\subseteq Q$. Taking $Q$ to be an appropriate subset of
$\overline{P}$, the chain completion of $P$, we can construct the
\emph{completion }of any semi-trim $P$-partition of an $\omega$-Stone
space. We define the \emph{clean interior }of any semi-trim partition
and examine the extent to which completion and clean interior are
``inverse'' operations.

\subsection{Extensions and the clean interior of a semi-trim partition}
\begin{defn}
Let $M,P$ and $Q$ be posets or PO systems with $M\subseteq P\subseteq Q$.
We say that $M$ is \emph{$P$-separated in $Q$} if $p=\sup_{Q}J$
for some $p\in M$ and ideal $J$ of $P$ implies $p\in J$; equivalently,
for countable $P$, if no $p\in M$ is the least upper bound in $Q$
of a strictly increasing sequence in $P$. We say that $M$ is \emph{separated
in $P$} if $M$ is $P$-separated in $P$.
\end{defn}

We start with a Proposition that will be used repeatedly in this section.
\begin{prop}
\label{Prop: Ideals}(Ideals) Let $P$ be a poset or PO system and
let $\mathscr{X}=\{X_{p}\mid p\in P\}^{*}$ be a semi-trim $P$-partition
of the $\omega$-Stone space $W$.

(i) If $J$ is an ideal of $\widehat{P}$ and $B\in\Trim(\mathscr{X})$
with $t(B)\in J$, then we can find $w\in B$ such that $I_{w}=J$;

(ii) If $P$ is a PO system and $\mathscr{X}$ is complete, $P^{d}$
is $\widehat{P}$-separated in $P$;

(iii) If $P$ is countable and $p\in P$, then $\{p\}$ is $\widehat{P}$-separated
in $P$ iff $\{p\}$ is separated in $P$.
\end{prop}

\begin{proof}
(i) By STP6, $\widehat{P}$ is countable, so we can find an ascending
chain $p=p_{1}\leqslant p_{2}\leqslant\cdots$ of elements of $\widehat{P}$
such that $J=\{q\in\widehat{P}\mid q\leqslant p_{n}\text{ for some }n\}$.
Let $R$ be the Boolean ring of compact open subsets of $W$, and
find finite subrings $R_{1}\leqslant R_{2}\leqslant\cdots$ of $R$
such that $B\in R_{1}$ and $R=\bigcup_{n\geqslant1}R_{n}$. Choose
$\{B_{n}\mid n\geqslant1\}$ successively such that $B=B_{1}\supseteq B_{2}\supseteq\cdots$
and $t(B_{n})=p_{n}$, by choosing an atom $C_{n}$ of $R_{n}$ that
meets $B_{n-1}\cap X_{p_{n}}$ (as $B_{1}\in R_{n}$ and $p_{n}\in T(B_{n-1})$)
and a $p_{n}$-trim subset $B_{n}$ of $B_{n-1}\cap C_{n}$, using
semi-trimness. Then $\bigcap_{n\geqslant1}B_{n}$ is a singleton set
$\{w\}$, say, as it is non-empty by compactness, and if $y\in B_{1}$
with $y\neq w$, then $y$ and $w$ belong to different atoms of $R_{n}$
for some $n$, and so $y\notin B_{n}$. It follows that $\{B_{n}\}$
is a neighbourhood base of $w$. But $I_{w}$ is an ideal of $\widehat{P}$
(STP1), and if $D\in I_{w}$, then $D\supseteq B_{n}$ for some $n$
and so $t(D)\leqslant p_{n}$. Hence $I_{w}=J$, as required.

(ii) Suppose $q=\sup_{P}J$ for some $q\in P$ and ideal $J$ of $\widehat{P}$,
with $q\notin J$. Use (i) to find $w\in W$ such that $I_{w}=J$.
By STP1A and completeness, $w\in X_{q}$ and so $w$ is a limit point.
Hence by STP4 $q\notin P^{d}$, and the result follows.

(iii) If $P$ is countable, $\{p\}$ is $\widehat{P}$-separated in
$P$ and $p=\sup_{P}J$ for some ideal $J$ of $P$, apply Lemma~\ref{lem:ideal chain}
with $M=\widehat{P}$ and $Q=P$, using STP2, to find an ideal $K$
of $\widehat{P}$ such that $K\subseteq J$ and $p=\sup_{P}K$; hence
$p\in K\subseteq J$. Conversely, if $\{p\}$ is separated in $P$
and $p=\sup_{P}K$ for some ideal $K$ of $\widehat{P}$, then $K_{\downarrow}$
is easily seen to be an ideal of $P$, so $p\in K_{\downarrow}$ and
hence $p\in K$. 
\end{proof}
\begin{lem}
\label{lem:ideal chain}Let $M,P$ and $Q$ be posets or PO systems
such that $M\subseteq P\subseteq Q$ and $Q$ is countable, and such
that (a) for all $p\in Q$, we can find $m\in M$ such that $m\leqslant p$,
(b) whenever $m\in M$ and $p\in P$ with $m\leqslant p$, there is
an ideal $K$ of $M$ such that $m\in K$ and $p=\sup_{Q}K$. 

Suppose $q\in Q$ and $q=\sup_{Q}J$ for an ideal $J$ of $P$. Then
there is an ideal $K$ of $M$ such that $K\subseteq J$ and $q=\sup_{Q}K$. 
\end{lem}

\begin{proof}
Find $p_{1}\leqslant p_{2}\leqslant p_{3}\leqslant\cdots$, with $p_{n}\in J\subseteq P$
and $q=\sup_{Q}\{p_{n}\}$. We need to find an ideal $K$ of $M$
such that $q=\sup_{Q}K$. If $q\in P$, this follows from the assumptions,
so we may assume that $q\notin P$.

Enumerate $Q$ as $\{s_{j}\mid j\geqslant1\}$ and let $L_{n}=\{j\leqslant n\mid s_{j}\ngeqslant p_{n}\}$.
Choose $m_{1}\in M$ with $m_{1}\leqslant p_{1}$. Suppose we have
already chosen $m_{2},\ldots,m_{n-1}$ with $m_{n-1}\leqslant p_{n-1}$,
and choose an ideal $K_{n}$ of $M$ such that $m_{n-1}\in K_{n}$
and $p_{n}=\sup_{Q}K_{n}$. If $p_{n}\in M$, let $m_{n}=p_{n}$.
Otherwise, for each $j\in L_{n}$ we can find $r_{nj}\in K_{n}$ such
that $s_{j}\ngeqslant r_{nj}$; hence we can find $m_{n}\in K_{n}$
such that $m_{n-1}\leqslant m_{n}\leqslant p_{n}$ and $s_{j}\ngeqslant m_{n}$
for all $j\in L_{n}$. Let $K$ be the ideal of $M$ generated by
$\{m_{n}\}$, so that $K\subseteq J$; we claim that $q=\sup_{Q}K$.

Clearly $q\geqslant m_{n}$ for all $n$. Suppose $t\in Q$ and $t\ngeqslant q$;
say $t=s_{j}$. Then $t\ngeqslant p_{n}$ for some $n\geqslant j$,
so $j\in L_{n}$ and $t\ngeqslant m_{n}$. Hence $q=\sup_{Q}K$.
\end{proof}
\begin{defn}
Let $P$ and $Q$ be posets or PO systems, with $P\subseteq Q$. Let
$\mathscr{X}=\{X_{p}\mid p\in P\}^{*}$ and $\mathscr{Y}=\{Y_{q}\mid q\in Q\}^{*}$
be a $P$-partition and $Q$-partition respectively of the Stone space~$W$
of the Boolean ring~$R$, with type functions $T$ and $U$ respectively.
We say $\mathscr{Y}$ is a \emph{regular extension }of $\mathscr{X}$
if

(i) $X_{p}\subseteq Y_{p}$ for each $p\in P$; 

(ii) for all $A\in R$ and $q\in U(A)$, we can find $p\in T(A)$
such that $p\leqslant q$.

The \emph{clean interior }of $\mathscr{X}$, denoted $\mathscr{X}^{o}$,
is the partition $\{Z_{q}\mid q\in\widehat{P}\}^{*}$, where $Z_{q}=\{w\in X_{q}\mid w\text{ is a clean point\ensuremath{\}}}$
for $q\in\widehat{P}$.

$Q$ is \emph{chain-unique} \emph{over }$P$ if for any ideal $J$
of $P$ which has a supremum $q\in Q$, if $r\in P$ and $r\lneqq q$,
then $r\in J$.

$Q$ is \emph{$\omega$-complete over $P$ }if every ascending sequence
of elements of $P$ has a supremum in $Q$.
\end{defn}

\begin{rem}
If $Q$ is chain-unique over $P$, and $I$ and $J$ are ideals of
$P$ that share the same supremum $q\in Q$, then $I-\{q\}=J-\{q\}=\{r\in P\mid r\lneqq q\}$;
equivalently, for countable $P$, if $\{p_{n}\}$ and $\{q_{n}\}$
are strictly increasing sequences of elements of $P$ that share the
same supremum in $Q$, then $\{p_{n}\}$ and $\{q_{n}\}$ can be ``interleaved'':
i.e.\ for each $q_{n}$ there is an $m$ such that $q_{n}\leqslant p_{m}$,
and vice versa. 

For semi-trim partitions, the extension is regular iff $A\in\Trim_{p}(\mathscr{X})$
and $A\cap Y_{q}\neq\emptyset$ implies $q\geqslant t(A)$.
\end{rem}

\begin{example}
Not all extensions are regular. Let $P=\mathbb{N}_{+}$ and $Q=\mathbb{N}$,
as PO systems having the usual order, with $P^{d}=\emptyset$ and
$Q^{d}=\{0\}$. By Theorem~\ref{existence construction}, there is
a compact $\omega$-Stone space $W$ with a trim $Q$-partition $\mathscr{Y}=\{Y_{n}\mid n\in Q\}^{*}$
such that $Y_{0}=\{y_{0}\}$, say. Choose a $0$-trim $A_{1}$ such
that $A_{1}\nsupseteq Y_{1}$, and find a neighbourhood base $A_{1}\supseteq A_{2}\supseteq\cdots$
of $y_{0}$ of $0$-trim sets. Let $X_{n}=Y_{n}-A_{n}$ for $n\geqslant1$
and let $\mathscr{X}=\{X_{n}\mid n\in P\}^{*}$, which is a $P$-partition
as $X_{n}\subseteq\overline{X_{n+1}}$. Moreover, $\mathscr{X}$ is
a trim $P$-partition: as if $w\neq y_{0}$, then we can find $B\in\Trim_{n}(\mathscr{Y})$
with $w\in B$ and $B\cap A_{n}=\emptyset$ for some $n$, and for
$m\geqslant n$ $C\subseteq B$ is $m$-trim with respect to $\mathscr{Y}$
iff it is $m$-trim with respect to $\mathscr{X}$; while for $w=y_{0}$,
$A_{n}$ is $j_{n}$-trim with respect to $\mathscr{X}$, where $j_{n}=\min\{k\mid A_{n}\cap X_{k}\neq\emptyset\}\geqslant n+1$.
So $X_{n}\subseteq Y_{n}$ for all $n\geqslant1$; however $A_{n}$
is $0$-trim with respect to $\mathscr{Y}$ but $j_{n}$-trim with
respect to $\mathscr{X}$, so the extension is not regular.
\end{example}

\begin{lem}
\label{Chain unique lemma}Let $P$ be a poset or PO system and $\{X_{p}\mid p\in P\}^{*}$
a complete semi-trim $P$-partition of the $\omega$-Stone space $W$.
The following are equivalent:

(i) $P$ is chain-unique over $\widehat{P}$;

(ii) If $w\in X_{q}$ ($q\in P$), then $\{p\in\widehat{P}\mid p\lneqq q\}\subseteq I_{w}$.
\end{lem}

\begin{proof}
If $P$ is chain-unique over $\widehat{P}$ and $w\in X_{q}$, $p\lneqq q$
and $p\in\widehat{P}$, then $I_{w}$ is an ideal of $\widehat{P}$
with $\sup_{P}I_{w}=q$ (by STP1A), so $p\in I_{w}$. 

Now suppose (ii) holds, that $J$ is an ideal of $\widehat{P}$ with
supremum $q\in P$, and that $p\in\widehat{P}$, with $p\lneqq q$.
By Proposition~\ref{Prop: Ideals}(i), we can find $w\in W$ such
that $I_{w}=J$. By completeness and STP1A, we have $w\in X_{q}$,
and now by (ii), $p\in I_{w}=J$, as required.
\end{proof}
\begin{rem}
Property (ii) above is a form of homogeneity, saying that for any
$q\in P$ all limit points in $X_{q}$ have the same types of trim
neighbourhoods (this is always true for clean points, irrespective
of $P$).
\end{rem}

\begin{lem}
\label{Regular extension lemma}Let $P$ and $Q$ be posets or PO
systems such that $P\subseteq Q$, and let $\mathscr{X}$ be a semi-trim
$P$-partition and $\mathscr{Y}$ a $Q$--partition of the Stone
space $W$ such that $\mathscr{Y}$ is a regular extension of $\mathscr{X}$. 

Then $\mathscr{Y}$ is a semi-trim $Q$-partition of $W$, $\mathscr{X}$
and $\mathscr{Y}$ give rise to the same trim sets, $\widehat{Q}=\widehat{P}$,
and if $P$ and $Q$ are PO systems then also $Q^{d}=P^{d}$.
\end{lem}

\begin{proof}
Let $R$ be the Boolean ring underlying $W$, and let $T:R\rightarrow2^{P}$
and $U:R\rightarrow2^{Q}$ be the type functions of $\mathscr{X}$
and $\mathscr{Y}$ respectively. Suppose $A\in R$ and let $F=T(A)_{\min}$,
which is a finite foundation for $T(A)$ (Proposition~\ref{Compact means FF}(i)).
Then by regularity, $F$ is also a finite foundation for $U(A)$,
so $U(A)_{\min}=T(A)_{\min}$, and $U(A)\cap P=T(A)$ as $U(A)$ and
$T(A)$ are upper subsets of $Q$ and $P$ respectively (Proposition~\ref{basic properties}(i)). 

If $P$ and $Q$ are posets, then $A\in\Trim_{p}(\mathscr{X})$ iff
$T(A)_{\min}=\{p\}$, and similarly for $\mathscr{Y}$; hence $\widehat{Q}=\widehat{P}$
and $\Trim_{p}(\mathscr{Y})=\Trim_{p}(\mathscr{X})$ for all $p\in\widehat{P}$. 

Suppose instead that $P$ and $Q$ are PO systems. If $q\in Q^{d}$
and $A\in\Trim_{q}(\mathscr{Y})$, then $q\in T(A)\cap\widehat{P}$
and so $A\cap X_{q}$ must be a singleton as $X_{q}\subseteq Y_{q}$;
hence $q\in P^{d}$ and $A\in\Trim_{q}(\mathscr{X})$. Conversely,
if $p\in P^{d}$ and $A\in\Trim_{p}(\mathscr{X})$, with $A\cap X_{p}=\{x\}$,
say, then we must also have $A\cap Y_{p}=\{x\}$, as otherwise we
could find $B\subseteq A-\{x\}$ such that $p\in U(B)$ but $p\notin T(B)$,
which is impossible. Hence $P^{d}=Q^{d}$ and $\Trim_{p}(\mathscr{Y})=\Trim_{p}(\mathscr{X})$
for all $p\in\widehat{P}$.

It remains to check that $\mathscr{Y}$ is a semi-trim $Q$-partition
of $W$. Semi-trim properties ST1 and ST2 are immediate, as $X_{p}\subseteq Y_{p}$
and $\mathscr{X}$ is semi-trim. For ST3, suppose $A\in R$ and $p\in\widehat{P}\cap U(A)$;
we must find a clean point in $A\cap Y_{p}$. But $p\in T(A)$, so
as $\mathscr{X}$ is semi-trim, we can find $x\in B\subseteq A$ such
that $B\in\Trim_{p}(\mathscr{X})$ and $x\in X_{p}\subseteq Y_{p}$,
as required. 
\end{proof}
\begin{prop}
\label{Clean interior propn}Let $P$ be a poset or PO system, and
let $\mathscr{X}$ be a semi-trim $P$-partition of the Stone space~$W$.
Then $\mathscr{X}$is a regular extension of $\mathscr{X}^{o}$, and
$\mathscr{X}^{o}$ is a trim $\widehat{P}$-partition of $W$ which
gives rise to the same trim sets as $\mathscr{X}$. 
\end{prop}

\begin{proof}
Let $R$ be the Boolean ring corresponding to $W$, let $\mathscr{X}=\{X_{p}\mid p\in P\}^{*}$,
and let $\mathscr{X}^{o}=\{Y_{q}\mid q\in\widehat{P}\}^{*}$ be the
clean interior of $\mathscr{X}$, noting that each $Y_{q}\neq\emptyset$.
For $q\in\widehat{P}$, $Y_{q}$ is dense in $X_{q}$ as $\mathscr{X}$
is semi-trim, and so $Y_{q}^{\prime}=X_{q}^{\prime}$ and $\overline{Y_{q}}=\overline{X_{q}}$,
from which it follows immediately that $\mathscr{X}^{o}$ is a $\widehat{P}$-partition.
Let $T:R\rightarrow2^{P}$ be the type function with respect to $\mathscr{X}$. 

If $A\in\Trim_{p}(\mathscr{X})$ then $A\in\Trim_{p}(\mathscr{X}^{o})$,
as $Y_{q}$ is dense in $X_{q}$ for $q\in\widehat{P}$, and if $P$
is a PO system and $p\in P^{d}$ then $|A\cap Y_{p}|=|A\cap X_{p}|=1$.

Conversely, suppose that $A\in\Trim_{p}(\mathscr{X}^{o})$. If $w\in A\cap X_{q}$,
find $B\in\Trim(\mathscr{X})$ such that $w\in B\subseteq A$; then
$B\cap X_{t(B)}\subseteq Y_{t(B)}$, so $q\geqslant t(B)\geqslant p$.
Hence $A\in\Trim_{p}(\mathscr{X})$, as $p\in T(A)$ and if also $P$
is a PO system and $p\in P^{d}\subseteq\widehat{P}$, then $X_{p}=Y_{p}$
as all points in $X_{p}$ are clean (STP5), and so $|A\cap X_{p}|=|A\cap Y_{p}|=1$.

Hence for $p\in\widehat{P}$, $\Trim_{p}(\mathscr{X})=\Trim_{p}(\mathscr{X}^{o})$
and all points in $Y_{p}$ are clean with respect to $\mathscr{X}^{o}$.
If now $w\in W$ has a neighbourhood base of sets in $\Trim_{q}(\mathscr{X}^{o})=\Trim_{q}(\mathscr{X})$
for $q\in\widehat{P}$, then $w\in X_{q}$ as $\mathscr{X}$ is full;
so $w$ is a clean point and $w\in Y_{q}$; thus $\mathscr{X}^{o}$
is full and so is a trim $\widehat{P}$-partition of $W$. If finally
$A\in R$ and $q\in T(A)$, then we can find $p\leqslant q$ with
$p\in\widehat{P}\cap T(A)$, so $A\cap Y_{p}\neq\emptyset$ and the
extension is regular, as required.
\end{proof}
We can now give necessary and sufficient conditions for when one semi-trim
partition of an $\omega$-Stone space can be regularly extended to
another semi-trim partition, which may or may not be complete.
\begin{thm}
\label{Completion thm}Let $P$ and $Q$ be posets or PO systems with
$P\subseteq Q$, and with $Q^{d}=P^{d}$ if $P,Q$ are PO systems.
Let $\mathscr{X}$ be a semi-trim $P$-partition of the $\omega$-Stone
space $W$. Then $\mathscr{X}$ can be regularly extended to a semi-trim
$Q$-partition of $W$ iff $Q$ satisfies the following 2 conditions:

(i) for each $q\in Q$, we can find $p\in\widehat{P}$ with $p\leqslant q$; 

(ii) if $p\in\widehat{P}$ and $q\in Q$ with $p\lneqq q$, then there
is an ideal $J$ of $\widehat{P}$ such that $p\in J$ and $q=\sup_{Q}J$.

Moreover, $\mathscr{X}$ can be regularly extended to a complete semi-trim
$Q$-partition $\mathscr{Y}$ of $W$ iff $Q$ satisfies the additional
conditions:

(iii) $Q$ is $\omega$-complete over $\widehat{P}$, and

(iv) $P^{d}$ is $\widehat{P}$-separated in $Q$, if $P$ is a PO
system.

In this case, $\mathscr{Y}$ will be a strongly semi-trim $Q$-partition
iff in addition $\mathscr{X}$ is a strongly semi-trim $P$-partition
and $\widehat{P}$ is $\widehat{P}$-separated in $Q$.
\end{thm}

\begin{proof}
Let $R$ be the Boolean ring corresponding to $W$, and $\mathscr{X}=\{X_{p}\mid p\in P\}^{*}$.
We note that $\widehat{P}$ is countable (STP6).

\textbf{Suppose first }that $Q$ satisfies conditions (i) and (ii).
We must regularly extend $\mathscr{X}$ to a semi-trim $Q$-partition
of $W$. Let $M$ be a subset of $\widehat{P}$, with $M\supseteq P^{d}$
if $P$ is a PO system; the purpose of $M$ is to avoid creating ``unwanted''
limit points in the expanded partition. In this proof, $V_{w}$ will
denote $\{A\in\Trim(\mathscr{X})\mid w\in A\}$.

For $w\in W$, let $\upsilon(w)=\sup_{Q}I_{w}$ if this supremum exists,
so $\upsilon(w)=p$ if $w\in X_{p}$ for $p\in P$ (by~STP1A). Let
$Y_{q}=\{w\in W\mid\upsilon(w)=q\}$ for $q\notin M$, and let $Y_{q}=X_{q}$
for $q\in M$: this gives a partition $\mathscr{Y}=\{Y_{q}\mid q\in Q\}^{*}$
of $W$, subject to confirming that each $Y_{q}\neq\emptyset$, with
$X_{p}\subseteq Y_{p}$ for $p\in P$. 

We first show (a) if $B\in\Trim_{p}(\mathscr{X})$ and $q\geqslant p$,
then $B\cap Y_{q}\neq\emptyset$; (b) $\mathscr{Y}$ is a $Q$-partition
of $W$. It follows from (a) that $B\cap Y_{q}\neq\emptyset$ for
all $q\in Q$ (find $p\in\widehat{P}$ with $p\leqslant q$ and choose
$B\in\Trim_{p}(\mathscr{X})$).

For (a), if $q\in P$, then $B\cap Y_{q}\supseteq B\cap X_{q}\neq\emptyset$
as $B$ is $p$-trim. If $q\notin P$, use (ii) to find an ideal $J$
of $\widehat{P}$ with $p\in J$ and $q=\sup_{Q}J$. Choose $w\in B$
such that $I_{w}=J$ (Proposition~\ref{Prop: Ideals}(i)). Then $\upsilon(w)=q\notin M$
and so $B\cap Y_{q}\neq\emptyset$.

For (b), choose $p\in Q$. If $q\in Q$, $q\leqslant p$, and $w\in Y_{q}$,
then for all $A\in V_{w}$ we have $t(A)\leqslant q\leqslant p$ (definition
of $Y_{q}$), so $A\cap Y_{p}\neq\emptyset$ by (a); hence $w\in\overline{Y_{p}}$
and $Y_{q}\subseteq\overline{Y_{p}}$. If $q\nleqslant p$ $(q\in Q$)
and $w\in Y_{q}$, then we can find $A\in V_{w}$ such that $t(A)\nleqslant p$,
as $q=\sup_{Q}I_{w}$; so $A\cap Y_{p}=\emptyset$ (definition of
$Y_{p}$) and $x\notin\overline{Y_{p}}$. 

If $P$ and $Q$ are PO systems, we must also show that $Y_{q}$ is
discrete if $q\in Q^{d}$ and otherwise has no isolated points. If
$q\in Q^{d}=P^{d}$, then $Y_{q}=X_{q}$ as $P^{d}\subseteq M$, so
$Y_{q}$ is discrete. Suppose $q\notin P^{d}$ and $y\in Y_{q}$.
Pick $A\in V_{y}$, so that $p\leqslant q$ where $p=t(A)\in P$,
and pick a clean point $x\in A\cap X_{p}$; we can ensure $x\neq y$,
as if $q=p$ then $X_{q}$ has no isolated points. Now choose $B\in V_{x}$
such that $B\subseteq A-\{y\}$; then $B\in\Trim_{p}(\mathscr{X})$
and so by (a) $B\cap Y_{q}\neq\emptyset$. Hence $y$ is not isolated
in $Y_{q}$, and $\mathscr{Y}$ is a $Q$-partition of $W$.

So by Lemma~\ref{Regular extension lemma}, $\mathscr{Y}$ is a semi-trim
$Q$-partition of $W$, $\mathscr{Y}$ and $\mathscr{X}$ give rise
to the same trim sets, and by definition of $\upsilon(w)$ $\mathscr{Y}$
regularly extends $\mathscr{X}$.

We note that $\mathscr{X}$ and $\mathscr{Y}$ have the same clean
points. For suppose $p\in\widehat{P}$ and $y\in Y_{p}$ is a clean
point in $\mathscr{Y}$. Then $y$ has a neighbourhood base of sets
in $\Trim_{p}(\mathscr{Y})=\Trim_{p}(\mathscr{X})$, and so $y\in X_{p}$
as $\mathscr{X}$ is full. Hence also all points in $Y_{p}-X_{p}$
are limit points (in $\mathscr{X}$ or $\mathscr{Y}$) for $p\in\widehat{P}$.

\textbf{Now suppose }that $\mathscr{X}$ can be regularly extended
to a semi-trim $Q$-partition $\mathscr{Y}$ of~$W$. By Lemma~\ref{Regular extension lemma},
$\mathscr{X}$ and $\mathscr{Y}$ give rise to the same trim sets
and $\widehat{Q}=\widehat{P}$. Moreover, $Q$ satisfies properties
(i) and (ii) by STP2.

\textbf{Next }if $Q$ satisfies properties (iii) and (iv), then $\upsilon(w)$
is defined for all $w\in W$ (Remark~\ref{rem:Ideals}). If $P$
is a PO system, choose $M=P^{d}$ in the construction; then for $p\in P^{d}$,
$\upsilon(w)=p$ iff $w\in X_{p}$, so $\mathscr{Y}$ is complete.
Conversely, if $\mathscr{Y}$ is complete and $p_{1}\leqslant p_{2}\leqslant p_{3}\leqslant\cdots$
is an ascending sequence in $\widehat{P}$, let $J=\{q\in P\mid q\leqslant p_{n}\text{ for some }n\}$
and use Proposition~\ref{Prop: Ideals}(i) to find $w\in W$ such
that $I_{w}=J$. Then if $w\in Y_{q}$, we have $q=\sup_{Q}I_{w}=\sup_{Q}\{p_{n}\mid n\geqslant1\}$
by STP1A. Hence $Q$ satisfies property (iii), and property (iv)
follows from Proposition~\ref{Prop: Ideals}(ii), as required.

Finally, if $\mathscr{Y}$ is strongly semi-trim, then so is $\mathscr{X}$,
and $Y_{p}=X_{p}$ for all $p\in\widehat{P}$, as $\mathscr{X}$ and
$\mathscr{Y}$ have the same clean points by the note above. So if
$p\in\widehat{P}$ and $p=\sup_{Q}J$ for some ideal $J$ of $\widehat{P}$,
find $w\in W$ such that $I_{w}=J$; then $w\in Y_{p}=X_{p}$, and
so $p\in J$ as $\mathscr{X}$ is strongly semi-trim. Conversely,
if $\mathscr{X}$ is strongly semi-trim and $\widehat{P}$ is $\widehat{P}$-separated
in $Q$ then let $M=\widehat{P}$ in the construction; for $p\in\widehat{P}$,
$\upsilon(w)=p$ iff $w\in X_{p}$, so $\mathscr{Y}$ is complete
with $X_{p}=Y_{p}$ for $p\in\widehat{P}$, and all points in $Y_{p}$
are clean.
\end{proof}
\begin{cor}
\label{Completion corollary 2}Let $P$ be a poset or PO system and
let $\mathscr{X}$ be a semi-trim $P$-partition of the $\omega$-Stone
space $W$. 

(i) $\mathscr{X}$ can be regularly extended to a complete semi-trim
$P$-partition of $W$ iff (a) $P$ is $\omega$-complete over $\widehat{P}$
and (b) $P^{d}$ is $\widehat{P}$-separated in $P$ if $P$ is a
PO system;

(ii) $\mathscr{X}$ can be regularly extended to a complete strongly
semi-trim $P$-partition of $W$ iff also $\mathscr{X}$ is strongly
semi-trim and $\widehat{P}$ is $\widehat{P}$-separated in $P$.
\end{cor}

\begin{proof}
(i) $\Leftarrow$ By Theorem~\ref{Completion thm} (taking $Q=P$)
$\mathscr{X}$ can be regularly extended to a complete semi-trim $P$-partition
of $W$, as conditions (i) and (ii) of that Theorem follow from semi-trimness
and STP2. 

$\Rightarrow$ Is immediate from Theorem~\ref{Completion thm}, taking
$Q=P$.

(ii) Immediate from Theorem~\ref{Completion thm}.
\end{proof}
\begin{rem}
The restriction for PO systems in the case of (ii) that $\widehat{P}$
be $\widehat{P}$-separated in $P$ is onerous, as all elements of
$P-\widehat{P}$ and no elements of $\widehat{P}$ must be the supremum
of some ideal of $\widehat{P}$.

If the trim $P$-partition $\mathscr{X}$ of an $\omega$-Stone space
$W$ can be regularly extended to a complete trim $P$-partition of
$W$, then $P$ has the ACC by Theorem~\ref{Trim nec condition},
so $\mathscr{X}$ is already a complete trim partition of $W$ (STP3).
\end{rem}

\subsection{The chain completion of a semi-trim partition}

We recall (Markowksy~\cite[Corollary~to~Theorem~6]{Markowsky}) that
we can extend any poset $P$ to its \emph{chain completion }$\overline{P}$
which has (inter alia) the following properties in the case when $P$
is countable; we are only considering non-empty chains here, so that
the chain completion $\overline{P}$ need not have a least element
corresponding to the supremum of the empty set:
\begin{lyxlist}{00.00.0000}
\item [{C1}] Every ascending sequence in $\overline{P}$ has a unique least
upper bound (supremum): i.e.\ $\overline{P}$ is $\omega$-complete
\item [{C2}] Every element of $\overline{P}$ is the supremum of an ideal
of $P$ (equivalently, of an ascending sequence $p_{1}\leqslant p_{2}\leqslant\cdots$,
with $p_{n}\in P$ for all $n$): i.e.\ $P$ is dense in $\overline{P}$
\item [{C3}] Completion preserves suprema: if $J$ is an ideal of $P$,
$q\in P$ and $q=\sup_{P}J$, then $q=\sup_{\overline{P}}J$
\item [{C4}] If $P$ is $\omega$-complete, then $\overline{P}\cong P$
\item [{C5}] If $q\in\overline{P}-P$ and $q=\sup_{\overline{P}}J$ for
some ideal $J$ of $P$, then $J=\{r\in P\mid r\lneqq q\}$.
\end{lyxlist}
Properties C1 to C4 are immediate from~\cite{Markowsky}. When $P$
is countable, it can be readily verified from the construction in~\cite{Markowsky}
that the elements of $\overline{P}$, viewed as subsets of $2^{P}$,
are the chain-closures (as defined in~\cite{Markowsky}) of non-empty
chains in $P$, each of which either has the form $\{q\in P\mid q\leqslant p\}$
for some $p\in P$, or is an ideal $J$ of $P$ that has no supremum
in $P$. Property C5 follows directly from this. This is a weaker
property than chain-uniqueness; in fact, $\overline{P}$ is chain-unique
over $P$ iff $P$ is chain-unique over itself.

We need two variations of this construction in what follows:
\begin{defn}
If $P$ is a countable PO system, we define the \emph{chain completion
}$\overline{P}$ of $P$ to be the PO system whose underlying poset
is the chain completion of the poset underlying $P$, and setting
$\overline{P}^{d}=P^{d}$ (i.e.\ $q<q$ for all $q\in\overline{P}-P$).

Let $P$ and $Q$ be posets or PO systems with $Q\subseteq P$ and
$Q$ countable. We define the \emph{chain completion of $P$ over
$Q$}, denoted $\overline{P:Q}$, to be the subset of $\overline{P}$
consisting of elements of $P$ together with the suprema in $\overline{P}$
of all ideals in $Q$, so that $\overline{P:Q}$ is $\omega$-complete
over $Q$. If $\mathscr{X}$ is a semi-trim $P$-partition of a topological
space, we write $\overrightarrow{P}$ for $\overline{P:\widehat{P}}$.
\end{defn}

\begin{prop}
Let $P$ be a countable poset or PO system and let $\mathscr{X}$
be a semi-trim $P$-partition of the Stone space $W$. If $P=\widehat{P}$
or if $\overline{P}$ is countable or if $P$ is chain-unique over
$\widehat{P}$, then $\overrightarrow{P}=\overline{P}$.
\end{prop}

\begin{proof}
The case $P=\widehat{P}$ (e.g.\ if $\mathscr{X}$ is a trim $P$-partition
of $W$) is immediate. 

Choose any $q\in\overline{P}$ and find $q_{1}\leqslant q_{2}\leqslant\cdots$,
with $q_{n}\in P$ and $q=\sup_{\overline{P}}\{q_{n}\}$ (property
C2 of $\overline{P}$). We need to find an ideal $K$ of $\widehat{P}$
such that $q=\sup_{\overline{P}}K$. If $q\in P$, this follows by
semi-trimness (STP1A and C3), so we may assume that $q\notin P$.

If $\overline{P}$ is countable, apply Lemma~\ref{lem:ideal chain}
with $M=\widehat{P}$ and $Q=\overline{P}$ to obtain the desired
ideal $K$, as the Lemma conditions follow from C2, STP2 and the
fact that suprema are preserved within $\overline{P}$.

Suppose lastly that $P$ is chain-unique over $\widehat{P}$. For
each $n$, let $K_{n}=\{p\in\widehat{P}\mid p\leqslant q_{n}\}$,
which is an ideal of $\widehat{P}$ with $q_{n}=\sup_{P}K_{n}$ (STP1A).
Let $K=\bigcup_{n\geqslant1}K_{n}$, which is easily seen to be an
ideal of $\widehat{P}$ as $K_{1}\subseteq K_{2}\subseteq\cdots$,
and let $u=\sup_{\overline{P}}K$ (using C1). Then $u\leqslant q$,
as $p\lneqq q$ for each $p\in K$. But $u\geqslant\sup_{\overline{P}}K_{n}=\sup_{P}K_{n}=q_{n}$
for all $n$ and hence $q=u=\sup_{\overline{P}}K$.
\end{proof}
\begin{cor}
\label{Completion corollary}Let $P$ be a poset and let $\mathscr{X}$
be a semi-trim $P$-partition of the $\omega$-Stone~space~$W$.
Then $\mathscr{X}$ can be regularly extended to a complete semi-trim
$\overrightarrow{P}$-partition of~$W$. 
\end{cor}

\begin{proof}
It is enough to show that $\overrightarrow{P}$ satisfies conditions
(i) to (iii) of Theorem~\ref{Completion thm}. Condition (iii) is
immediate from property C1 of $\overline{P}$ and the definition of
$\overrightarrow{P}$. Condition (i) follows from the definition of
$\overrightarrow{P}$, together with the observation that if $p\in P$,
then by semi-trimness of $\mathscr{X}$ we can find $q\in\widehat{P}$
with $q\leqslant p$. It remains to show that condition (ii) holds.

Suppose then that $p\in\widehat{P}$ and $q\in\overrightarrow{P}$
with $p\lneqq q$. If $q\in P$, then take $B$ to be any $p$-trim
set, choose $w\in B\cap X_{q}$ and let $J=I_{w}$, so that $q=\sup_{P}I_{w}=\sup_{\overrightarrow{P}}I_{w}$.
Otherwise, we can find an ideal $J$ of $\widehat{P}$ with $q=\sup_{\overrightarrow{P}}J=\sup_{Q}J$,
where $Q=\overline{\widehat{P}}$. Now $\widehat{P}$ is countable
by STP6, and so property C5 of the chain completion $Q$ applies
(as $q\notin P)$, to give $p\in J$.
\end{proof}
We will say that the $\overrightarrow{P}$-partition of $W$ created
above is the \emph{chain completion} of the semi-trim $P$-partition
$\mathscr{X}$ and we will denote it by $\overline{\mathscr{X}}$.
The partition $\overline{\mathscr{X}}$ will be strongly semi-trim
iff $\mathscr{X}$ is a strongly semi-trim $P$-partition of $W$
and $\widehat{P}$ is $\widehat{P}$-separated in $P$ (which by C3
means that $\widehat{P}$ is also $\widehat{P}$-separated in $\overrightarrow{P}$).

If $P$ is a PO system, Corollary~\ref{Completion corollary} still
holds provided $P^{d}$ is $\widehat{P}$-separated in $P$.
\begin{question}
Can we find countable posets $P\subseteq Q$ satisfying conditions
(i) and (ii) of Theorem~\ref{Completion thm}, letting $\widehat{P}=P$,
such that $\overrightarrow{Q}\neq\overline{Q}$, where $\overrightarrow{Q}=\overline{Q:P}$?
For then we could find an $\omega$-Stone space $W$ with a trim $P$-partition
(Theorem~\ref{Existence theorem}), which could be regularly extended
first to a semi-trim $Q$-partition (Theorem~\ref{Completion thm})
and then to a complete semi-trim $\overrightarrow{Q}$-partition of
$W$ (Corollary~\ref{Completion corollary}). This would be an example
of a complete $\overrightarrow{Q}$-partition of an $\omega$-Stone
space where the underlying poset $\overrightarrow{Q}$ was not $\omega$-complete
(cf~Theorem~\ref{Trim nec condition}).
\end{question}

The next Theorem shows the quasi-inverse nature of the chain completion
and clean interior of a partition.
\begin{thm}
Let $P$ and $Q$ be posets, and let $W$ be an $\omega$-Stone space.

(i) A trim $P$-partition $\mathscr{X}$ of $W$ is the clean interior
of its chain completion: $\mathscr{X}=\mathscr{\overline{X}}^{o}$.

(ii) If $\mathscr{Y}$ is a complete semi-trim $Q$-partition of $W$
and $Q$ is chain-unique over $\widehat{Q}$, then $\mathscr{Y}$
refines $\overline{\mathscr{Y}^{o}}$.
\end{thm}

\begin{rem}
This also holds if $P$ and $Q$ are PO systems, provided $P^{d}$
is $\widehat{P}$-separated in $P$ and similarly for $Q$. 
\end{rem}

\begin{proof}
For (i), let $\mathscr{Y}=\overline{\mathscr{X}}=\{Y_{q}\mid q\in\overrightarrow{P}\}$
be the chain completion of the trim $P$-partition $\mathscr{X}=\{X_{p}\mid p\in P\}^{*}$,
and let $\mathscr{Y}^{o}=\{Z{}_{r}\mid r\in\widehat{\overrightarrow{P}}\}^{*}$
be the clean interior of $\mathscr{Y}$. By Corollary~\ref{Completion corollary},
Lemma~\ref{Regular extension lemma} and Proposition~\ref{Clean interior propn},
$\mathscr{X}$, $\mathscr{Y}$ and $\mathscr{Y}^{o}$ all give rise
to the same trim sets, so $\widehat{\overrightarrow{P}}=P$. 

If $p\in P$ and $x\in X_{p}$, then $x$ is a clean point of $Y_{p}$,
so $x\in Z_{p}$ and $X_{p}\subseteq Z_{p}$. Conversely, if $r\in P$
and $z\in Z_{r}$, then $z$ has a neighbourhood base of $r$-trim
sets. As $\mathscr{X}$ is full, we have $z\in X_{r}$, $Z_{r}=X_{r}$,
and so $\mathscr{X}$ is the clean interior of its chain completion.

For (ii), let $\mathscr{Y}=\{Y_{q}\mid q\in Q\}$ be a complete semi-trim
$Q$-partition of $W$, whose clean interior is $\mathscr{X}=\mathscr{Y}^{o}=\{X_{p}\mid p\in P\}^{*}$,
where $P=\widehat{Q}$, which by Proposition~\ref{Clean interior propn}
is a trim $P$-partition of $W$. Let $\mathscr{Z}=\overline{\mathscr{X}}=\{Z{}_{r}\mid r\in\overrightarrow{P}\}$
be the chain completion of $\mathscr{X}$. By Proposition~\ref{Clean interior propn},
Lemma~\ref{Regular extension lemma} and Corollary~\ref{Completion corollary},
the partitions $\mathscr{Y}$, $\mathscr{X}$ and $\mathscr{Z}$ all
give rise to the same trim sets, and hence give rise to the same clean
points as $\mathscr{X}$ is full. 

If $w\in Y_{q}$ ($q\in Q$), then $q=\sup_{Q}I_{w}=\sup_{Q}J_{q}$,
where $J_{q}=\{p\in P\mid p\leqslant q\}$, as by Lemma~\ref{Chain unique lemma}
applied to $\mathscr{Y}$ we have $\text{\ensuremath{\{p\in P\mid p\lneqq q\}\subseteq I_{w}}}$.
Moreover, $w\in Z_{r}$ where $r=\sup_{\overrightarrow{P}}I_{w}=\sup_{\overrightarrow{P}}J_{q}$:
as if $q\in P$, then $q=\sup_{Q}I_{w}=\sup_{P}I_{w}=\sup_{\overrightarrow{P}}I_{w}$,
because $\overrightarrow{P}\subseteq\overline{P}$ and chain completion
of a poset preserves suprema.

We can therefore define a map $\alpha:Q\rightarrow\overrightarrow{P}$
by setting $q\alpha=\sup_{\overrightarrow{P}}J_{q}$, with $Y_{q}\subseteq Z_{q\alpha}$.
An easy check shows that $\alpha$ is an order isomorphism (i.e.\ if
$q\leqslant r$ then $q\alpha\leqslant r\alpha$), and $\alpha$ is
surjective as $Z_{r}\neq\emptyset$ for each $r\in\overrightarrow{P}$. 

It follows that $\overrightarrow{P}=Q\alpha$ and that $Z_{r}=\bigcup\{Y_{q}\mid q\alpha=r\}$,
as $\mathscr{Y}$ and $\mathscr{Z}$ are complete partitions. Hence
$\mathscr{Y}$ refines $\mathscr{Z}$, as required.
\end{proof}
\begin{rem}
The map $\alpha:Q\rightarrow\overrightarrow{P}$ is ``nearly'' injective
on $Q-P$, as if $q,r\in Q-P$ and $q\alpha=r\alpha=s$ say, with
$s\notin P$, then by C5 we have $\{p\in P\mid p\lneqq q\}=\{p\in P\mid p\lneqq r\}$,
and hence (taking $\sup_{Q})$ that $q=r$.

For a general choice of $Q$, and using the notation in the proof
of the previous Theorem, the relationship between $\mathscr{Y}$ and
$\mathscr{Z}$ is more complex, with $Y_{p}\subseteq Z_{p}$ for $p\in P$,
and $Z_{r}\subseteq Y_{r\beta}$ for $r\notin P$ where $r\beta=\sup_{Q}\{p\in P\mid p\lneqq r\}$:
as if $w\in Z_{r}$ $(r\notin P)$ then $I_{w}=\{p\in P\mid p\lneqq r\}$,
using property C5 of $\overline{P}$ and the first half of the proof
of Lemma~\ref{Chain unique lemma}. The next two examples illustrate
what can go wrong when $Q$ is chain-unique over $\widehat{Q}$ and
$q\alpha=r\alpha\in P$, or when $Q$ is not chain-unique over $\widehat{Q}$. 
\end{rem}

\begin{example}
(``$q\alpha=r\alpha\in P$'') Example~\ref{exa:Ideal completion}
exhibited a Stone space with a trim $P$-partition $\mathscr{X}$,
where $P=\mathbb{N}\cup\{\infty\}$, which could be extended to a
complete semi-trim partition in (at least) two different ways:

(a) expand $X_{\infty}$ to include all limit points. This is the
chain completion $\mathscr{Z}=\overline{\mathscr{X}}$ of $\mathscr{X}$,
which is a $P$-partition (as $P$ is $\omega$-complete) but is not
strongly semi-trim as $Z_{\infty}$ will contain a mix of clean and
limit points;

(b) let $Q=P\cup\{r\}$, where $n\lneqq r\lneqq\infty$ for all $n$,
let $Y_{p}=X_{p}$ for $p\in P$ and let $Y_{r}=W-X_{P}$ (i.e.\ all
the limit points). The resulting partition $\mathscr{Y}$ will be
a complete strongly semi-trim $Q$-partition of $W$, with $\widehat{Q}=P$
and $\mathscr{Y}^{o}=\mathscr{X}$. For $\alpha:Q\rightarrow P$ defined
above, we have $\infty\alpha=r\alpha=\infty$, $Z_{\infty}=Y_{\infty}\dotplus Y_{r}$,
and $\mathscr{Y}$ refines $\overline{\mathscr{X}}$, as per the previous
Theorem.
\end{example}

If $Q$ is not chain-unique over $\widehat{Q}$, then the map $\alpha$
may cease to be well-defined. 
\begin{example}
Let $P$ be the poset $\{p_{1},p_{2},\ldots,q_{1},q_{2},\ldots\}$,
where $\{p_{n}\}$ and $\{q_{n}\}$ are increasing sequences. By Theorem~\ref{Existence theorem}
there is an $\omega$-Stone space $W$ with a trim $P$-partition,
which can be extended to a complete semi-trim $\overrightarrow{P}$-partition
$\mathscr{Z}=\{Z_{u}\mid u\in\overrightarrow{P}\}$ of $W$ by Corollary~\ref{Completion corollary}.
Now $\overrightarrow{P}=P\cup\{r,s\}$ where $r=\sup\{p_{n}\}$ and
$s=\sup\{q_{n}\}$, and $Z_{r}$ and $Z_{s}$ contain only limit points. 

Let $Q=P\cup\{t\}$, where $t=\sup\{p_{n}\}=\sup\{q_{n}\}$, and let
$Y_{p}=Z_{p}$ for $p\in P$ and $Y_{t}=Z_{r}\cup Z_{s}$; then $\{Y_{q}\mid q\in Q\}$
is a complete semi-trim $Q$-partition. Its clean interior will be
$\{Z_{p}\mid p\in P\}^{*}$ and the chain completion of its clean
interior will be $\{Z_{u}\mid u\in\overrightarrow{P}\}$. 

So in the previous example $Z_{\infty}$ splits into two partition
elements in $\mathscr{Y}$ (noting that $q\in P$), whereas in this
example $Y_{t}$ splits into two partition elements in $\mathscr{Z}$
(noting that $t\notin P$). 
\end{example}

\begin{rem}
If $P$ is a poset and $\mathscr{X}$ is a trim $P$-partition of
an $\omega$-Stone space $W$, we can consider the family $\mathscr{F}$
of semi-trim partitions $\mathscr{Y}$ of $W$ such that $\mathscr{Y}^{o}=\mathscr{X}$,
noting that $P$ is countable. All partitions in $\mathscr{F}$ share
the same trim sets, and every partition $\mathscr{Y}\in\mathscr{F}$
can be regularly extended to a complete semi-trim partition $\overline{\mathscr{Y}}\in\mathscr{F}$
(Corollary~\ref{Completion corollary}). There is a canonical complete
semi-trim member of $\mathscr{F}$, namely the chain completion $\overline{\mathscr{X}}$.
Other complete partitions in $\mathscr{F}$ may be refinements of
$\overline{\mathscr{X}}$ or may be related to $\overline{\mathscr{X}}$
in more complex ways, such as those illustrated in the two previous
examples.
\begin{rem}
This sub-section has looked at the chain completions of a poset or
PO system $P$, and of a semi-trim $P$-partition of a Stone space.
There is a second type of completion of a poset $P$, the \emph{ideal
completion}, whose elements are all the ideals of $P$, where (unlike
the chain completion) an ideal is not identified with its supremum.
Example~\ref{exa:Ideal completion} is the simplest example of a
trim partition where the chain completion and ideal completion are
different. The ideal completion does not preserve (indeed, deliberately
destroys) suprema, but has the advantages firstly that the ideal completion
of a trim partition (using Theorem~\ref{Completion corollary 2})
is strongly semi-trim, and secondly that no additional separation
restrictions are required when $P$ is a PO system. Remark~\ref{rem:Index partition}
shows another example of the ideal completion at work.
\end{rem}

\end{rem}

\section{\label{sec:Existence-Conditions}Existence conditions}

\subsection{Necessary conditions for existence of complete $P$-partitions}
\begin{thm}
\label{Trim nec condition}Let $P$ be a poset or PO system.

(i) A trim $P$-partition of a Stone space is complete iff $P$ has
the ACC;

(ii) If there is an $\omega$-Stone space that admits a complete $P$-partition,
and $P$ is countable, then $P$ is $\omega$-complete.
\end{thm}

\begin{rem}
Part (i) generalises a result proved in Pierce~\cite[3.11.2]{PierceMonk}.
\end{rem}

\begin{proof}
\textbf{(}i) If $P$ has the ACC then any trim $P$-partition is complete,
by STP3. Suppose conversely that the Stone space $W$ admits a complete
trim $P$-partition $\{X_{p}\mid p\in P\}$, and that $q_{1}\lneqq q_{2}\lneqq q_{3}\lneqq\cdots$
is an ascending chain in $P$. 

We can find $A_{1}\supseteq A_{2}\supseteq A_{3}\supseteq\cdots$
such that $A_{n}$ is $q_{n}$-trim, as we can find a $q_{1}$-trim
$A_{1}$, and if $A_{k}$ is $q_{k}$-trim then $A_{k}\cap X_{q_{k+1}}\neq\emptyset$,
so we can find a $q_{k+1}$-trim $A_{k+1}\subseteq A_{k}$. Let $A_{\infty}=\bigcap_{i\geqslant1}A_{i}$,
which is non-empty by compactness. For each $x\in A_{\infty}$, find
a $p_{x}$-trim open neighbourhood $B_{x}$ of $x$, where $x\in X_{p_{x}}$;
we note that $p_{x}\gneqq q_{n}$ for all $n$. By considering the
cover $\{A_{1}-A_{2},A_{2}-A_{3},\ldots,\{B_{x}\mid x\in A_{\infty}\}\}$
of the compact set $A_{1}$, we obtain a finite subset $E$ of $A_{\infty}$
such that $\bigcup_{x\in E}B_{x}\supseteq A_{n}$ for some $n$, and
so $q_{n}\geqslant p_{x}$ for some $x\in E$, which is a contradiction.
Hence $P$ has the ACC\@. 

(ii) Suppose that the Stone space $W$ of the countable Boolean ring
$R$ admits a complete $P$-partition $\{X_{p}\mid p\in P\}$, where
$P$ is countable. We must show that $P$ is $\omega$-complete.

Let $q_{1}\leqslant q_{2}\leqslant q_{3}\leqslant\cdots$ be an ascending
chain in $P$ and let $P=\bigcup_{n\geqslant1}P_{n}$, where $\{q_{1}\}=P_{1}\subseteq P_{2}\subseteq\cdots$
and each $P_{n}$ is finite. We claim that we can choose $A_{1}\supseteq A_{2}\supseteq A_{3}\supseteq\cdots$
such that for all $i$, $A_{i}\in R$ and

\begin{equation}
\begin{array}{c}
A_{i}\cap X_{q_{i}}\neq\emptyset\\
A_{i}\cap X_{p}=\emptyset\text{ for all }p\in P_{i}\text{\text{ such that }\ensuremath{p\ngeqslant q_{i}}}
\end{array}\Biggl\}\label{eq:Chain construction-1}
\end{equation}

To see this, choose $A_{1}\in R$ with $A_{1}\cap X_{q_{1}}\neq\emptyset$,
enumerate the subsets of $A_{1}$ in $R$ as $\{B_{i}\mid i\geqslant1\}$
with $B_{1}=A_{1}$, and let $R_{n}$ be the finite Boolean ring generated
by $\{B_{i}\mid1\leqslant i\leqslant n\}$. Suppose $n\geqslant2$
and we have already chosen $A_{2},\ldots,A_{n-1}$ satisfying (\ref{eq:Chain construction-1}).
To choose $A_{n}$, find $x_{n}\in A_{n-1}\cap X_{q_{n}}$ (as $q_{n}\in T(A_{n-1})$),
find $C_{n}\in R$ containing $x_{n}$ with $C_{n}\cap X_{p}=\emptyset$
for all $p\in P_{n}$ such that $p\ngeqslant q_{n}$ (i.e.\ such
that $\overline{X_{p}}\cap X_{q_{n}}=\emptyset$), and let $A_{n}=A_{n-1}\cap C_{n}\cap D_{n}$,
where $D_{n}$ is the atom of $R_{n}$ containing $x_{n}$. Then $A_{n}$
satisfies (\ref{eq:Chain construction-1}), as required.

Now let $A_{\infty}=\bigcap_{i\geqslant1}A_{i}$, which is non-empty
by compactness. Let $x_{\infty}$ be an element of $A_{\infty}$.
If now $y\neq x_{\infty}$ is any other element of $A_{1}$, then
we can find disjoint subsets $B_{j}$ and $B_{k}$ of $A_{1}$ containing
$x_{\infty}$ and $y$ respectively. Hence $x_{\infty}$ and $y$
will belong to different atoms of $R_{l}$, where $l=\max(j,k)$,
so $y\notin A_{l}$ and $y\notin A_{\infty}$. Hence $A_{\infty}=\{x_{\infty}\}$.

Find $p$ such that $x_{\infty}\in X_{p}$. We claim that $p$ is
the least upper bound of $\{q_{i}\}$. For suppose first that $r\geqslant q_{i}$
for all $i$. Then $\overline{X_{r}}\cap A_{i}\neq\emptyset$ for
all $i$, so by compactness $\overline{X_{r}}\cap A_{\infty}\neq\emptyset$;
hence $x_{\infty}\in\overline{X_{r}}$ and so $r\geqslant p$. Conversely,
we have $p\geqslant q_{i}$ for all $i$: as if $p\ngeqslant q_{i}$,
say, find $m\geqslant i$ such that $p\in P_{m}$; then $x_{\infty}\in A_{m}\cap X_{p}$,
but $p\ngeqslant q_{m}$, as $q_{m}\geqslant q_{i}$, and so by construction
$A_{m}\cap X_{p}=\emptyset$, which is a contradiction. It follows
that $P$ is $\omega$-complete. 
\end{proof}

\subsection{Construction of $P$-partitions}

We now turn to the construction of $P$-partitions, where $P$ is
a PO system; the next definitions provide a convenient way of describing
their isolated point and compactness structure.
\begin{defn}
An \emph{extended PO system }is a triple $[P,L,f]$, where $P$ is
a PO system, $L$ is a lower subset of $P_{\Delta}$, and $f:L_{\min}^{d}\rightarrow\mathbb{N}_{+}$,
where $L_{\min}^{d}=L_{\min}\cap P^{d}$. 

If $[P,L,f]$ is an extended PO system, an order automorphism $\theta$
of the underlying poset is a \emph{$[P,L,f]$-automorphism }if $L\theta=L$,
$P^{d}\theta=P^{d}$ and $f(p\theta)=f(p)$ for all $p\in L_{\min}^{d}$.

Let $W$ be a Stone space and $[P,L,f]$ an extended PO system. We
will say that a partition $\mathscr{X}=\{X_{p}\mid p\in P\}{}^{*}$
of $W$ is a \emph{$[P,L,f]$-partition of $W$ }if it is a $P$-partition
such that $\overline{X_{p}}$ is compact iff $p\in L$ and $|X_{p}|=f(p)$
if $p\in L_{\min}^{d}$; and that $\mathscr{X}$ is a \emph{bounded
$[P,L,f]$-partition }of $W$ if in addition $X_{L}$ is bounded,
where a subset $E$ of $W$ is \emph{bounded} if $\overline{E}$ is
compact and is otherwise \emph{unbounded}.
\end{defn}

\begin{rem}
If $p\in P^{d}-L_{\min}$ then $X_{p}$ is infinite, as $\overline{X_{p}}\supseteq X_{q}$
if $p\gneqq q$ (case $p\notin P_{\min}$) or $X_{p}$ is not compact
(case $p\notin L$); and by counting trim neighbourhoods, $X_{p}$
will be countable for such a $p$ if $R$ is countable. 

Different choices of $\{P,L,f,L^{u}\}$ in the following Theorem allow
the creation of $P$-partitions with varying isolated point and compactness
properties.
\end{rem}

\begin{thm}
\label{existence construction}Let $[P,L,f]$ be a countable extended
PO system and $L^{u}$ an upper subset of $L$. Let $L^{b}=L-L^{u}$.

Then there is a trim $[P,L,f]$-partition $\mathscr{X}=\{X_{p}\mid p\in P\}^{*}$
of an $\omega$-Stone space $W$ such that for any subset $Q\subseteq L$
that has a finite foundation but does not have a finite ceiling:

(i) if $Q\subseteq L^{b}$, then $X_{Q}$ is bounded;

(ii) if $Q$ is an upper subset of $L^{u}$, then $X_{Q}$ is unbounded.

$W$ will be compact iff $P$ has a finite foundation, $L=P$, and
$L^{u}$ has a finite ceiling.

The partition $\mathscr{X}$ will be complete iff $P$ has the ACC. 

If $P$ is $\omega$-complete, with $P^{d}$ separated in $P$ if
$P$ is a PO system, then $\mathscr{X}$ can be extended to a complete
semi-trim $[P,L,f]$-partition of $W$ satisfying (i) and (ii) above.
\end{thm}

\begin{rem}
\label{rem:Existence remark}For the same $L$, different choices
of $L^{u}$ may give rise to the same $[P,L,f]$-partition: for example,
if $L^{u}$ is changed by only finitely many elements. We note that
if $Q\subseteq L$ does not have a finite foundation, then $X_{Q}$
is always unbounded (Proposition~\ref{Compact means FF} applied
to $\overline{X_{Q}}$, as $T(\overline{X_{Q}})=Q_{\downarrow}$),
while if $Q$ has a finite ceiling and $\overline{X_{p}}$ is compact
for all $p\in Q$ then so is $\overline{X_{Q}}$.

Typical choices for $L^{u}$ as an upper subset of $L\subseteq P_{\Delta}$
are either an infinite antichain or an infinite ascending sequence.
\end{rem}

\begin{proof}
Our proof is an evolution of that for Theorem~7.1 of~\cite{AppsBP},
but ignoring the group action. Enumerate $P=\{p_{i}\mid i\geqslant1\}$,
let $P_{n}=\{p_{i}\mid i\leqslant n\}$ and let $M_{n}=\{p\in L^{u}\mid p\leqslant q\text{ for some }q\in P_{n}\cap L^{u}\}$.
For $p\in P$, let $g(p)=f(p)$ if $p\in L_{\min}^{d}$, otherwise
let $g(p)=1$.

\textbf{Step 1}. 

Let $Z(0)$ be the empty set. For each $n\geqslant1$, we can inductively
construct finite sets $Z(n)$, a partition 

\[
\{Z(n,A)\mid A\in Z(n-1)\}\cup U_{n}\text{ of }Z(n)
\]
and a type function $t:Z(n)\rightarrow P$ such that:

(a) if $A\in Z(n\lyxmathsym{\textemdash}1)$, then $Z(n,A)$ contains
exactly one element of type $q$ for each element of $\{q\in P_{n}\mid q\geqslant t(A)\}$,
except that if $n\geqslant2$ and $t(A)\notin P^{d}$ then $Z(n,A)$
contains exactly two elements of type $t(A)$;

(b) $U_{n}$ contains $g(p_{n})$ elements of type $p_{n}$ if either
(b1) there is no $A\in Z(n-1)$ such that $p_{n}\geqslant t(A)$,
or (b2) $p_{n}\in L^{u}-M_{n-1}$;

(c) for each element of $P_{n-1}-L$, $U_{n}$ also contains one element
of the corresponding type. 

It follows readily by induction that $P_{n}\subseteq\{t(B)\mid B\in Z(n)\}$.
We will say that elements of $Z(n,A)$ \emph{descend from $A$}, and
extend this transitively.

\textbf{Step 2}. 

We construct a countable Boolean ring $R$ from the $Z(n)$, in which
$\bigcup_{n\geqslant1}Z(n)$ will be a tree of trim sets, with $D\in Z(n)$
being $t(D)$-trim.

For each $n\geqslant1$, let $R_{n}$ be the ring of subsets of $Z(n)$,
whose atoms are the elements of $Z(n)$. Define an injection $\theta_{n}:R_{n-1}\rightarrow R_{n}$
by 

\[
A\theta_{n}=Z(n,A)\text{, for }A\in Z(n-1)
\]

Let $R$ be the direct limit of the system $R_{1}\rightarrow R_{2}\rightarrow R_{3}\rightarrow\cdots$,
which is a countable Boolean ring, and let $W$ be its Stone space.
For $A\in Z(n)$, write $\widetilde{A}$ for the image of $A$ in
$R$, viewed as a subset of $W$, and write $t(\widetilde{A})=t(A)$.
This is well-defined, as if $Z(n,A)$ is a singleton subset of $R_{n}$,
then $Z(n,A)$ contains a single point of type $t(A)$, and so $t(A\theta_{n})=t(A)$.
We note:

\begin{equation}
\text{if }A\in R_{n}\text{ and }B\in U_{m}\text{ }(m>n)\text{, then }A\cap\widetilde{B}=\emptyset.\label{eq:Um}
\end{equation}

Let $S=\bigcup_{n\geqslant1}\widetilde{Z(n)}$; we claim that $t:S\rightarrow P$
structures $R$. Conditions S1 and S2 of Definition~\ref{def:structure}
are immediate by construction, as if $A\in\widetilde{Z(n)}$ and $B\in\widetilde{Z(n+1)}$
with $B\subseteq A$ then $t(B)\geqslant t(A)$. For S3, if $A,B,C\in S$,
$A\supseteq B\dotplus C$ and $t(A)=t(B)=t(C)=p$, then by considering
descendants of $B$ and $C$ we may assume that $A\in\widetilde{Z(n)}$
and $B,C\in\widetilde{Z(m)}$ for some $m>n$; but if $p\in P^{d}$
then for all $m>n$ there is just one point of type $p$ amongst the
partition elements of $Z(m)$ that descend from $A$, and so $p\notin P^{d}$
and $p<p$. For S4, if $A=B\dotplus C$ with $A\in\widetilde{Z(n)}$,
find $m>n$ such that $B,C\in R_{m}$; then by construction there
will be an atom $D\subseteq A$ of $R_{m}$ such that $t(D)=t(A)$,
and either $D\subseteq B$ or $D\subseteq C$. Finally for S5, if
$A\in\widetilde{Z(n)}$ and $t(A)<q$, find $m>n$ such that $q\in P_{m}$,
and find $D\in\widetilde{Z(m-1)}$ such that $t(D)=t(A)$ and $D\subseteq A$;
then by construction we can find $B,C\in\widetilde{Z(m)}$ such that
$B,C\subseteq D$, $t(B)=t(A)$ and $t(C)=q$, as if $q=t(A)$ then
$t(A)\notin P^{d}$. 

So by Proposition~\ref{structure=00003Dtrim} we can find a trim
$P$-partition $\mathscr{X}=\{X_{p}\mid p\in P\}^{*}$ of $W$ such
that each $A\in S$ is $t(A)$-trim. We must show that $\mathscr{X}$
satisfies (i) and (ii) and is a trim $[P,L,f]$-partition.

\textbf{Step 3}. 

For (i), we will establish the following stronger statement:

\begin{equation}
\begin{array}{c}
\text{if the subset }Q\text{ of }L\text{ has a finite foundation}\\
\text{and \ensuremath{Q\cap L^{u}} has a finite ceiling, then }X_{Q}\text{ is bounded}
\end{array}\Biggl\}\label{eq:XQ compact}
\end{equation}

Replace $Q$ with $Q_{\downarrow}\subseteq L$ so that $Q$ is a lower
subset of $L$. Let $F$ be a finite foundation for $Q$ and $G$
a finite ceiling for $Q\cap L^{u}$, and find $n$ such that $F\cup G\subseteq P_{n}$,
so that $M_{n}=L^{u}$. Let $A=\bigcup_{B\in Z(n)}\widetilde{B}$,
which is compact. We claim that if $m\geqslant n$, $C\in Z(m)$ and
$t(C)\in Q$ then $\widetilde{C}\subseteq A$. 

This is clearly true for $m=n$. Suppose it is true for $m=k\geqslant n$,
and let $C\in Z(k+1)$ with $t(C)\in Q$. Now $t(C)\geqslant r$ for
some $r\in F$, and $F\subseteq P_{k}$, so there is a $B\in Z(k)$
with $t(B)=r$. By construction, $C\notin U_{k+1}$: it fails condition
(b1) as $t(C)\geqslant t(B)$ with $B\in Z(k)$, condition (b2) as
$M_{k}=L^{u}$ and condition (c) as $t(C)\in L$. Hence $\widetilde{C}\subseteq\widetilde{D}$
for some $D\in Z(k)$ with $t(D)\leqslant t(C)\in Q$; so $t(D)\in Q$
and by the induction hypothesis we have $\widetilde{C}\subseteq\widetilde{D}\subseteq A$.
The claim follows by induction. 

Suppose finally that $x\in X_{q}$ ($q\in Q$). Find $C\in Z(m)$
for some $m>n$ such that $x\in\widetilde{C}$. Then $t(\widetilde{C})\leqslant q$
and so $\widetilde{C}\subseteq A$. Hence $X_{Q}\subseteq A$, as
required. Statement (i) now follows.

For (ii), if $Q$ is an upper subset of $L^{u}$ that has a finite
foundation but does not have a finite ceiling, pick $A\in R$ and
suppose $A\in R_{n}$. Now $Q\nsubseteq M_{n}$ as otherwise $P_{n}\cap Q$
would be a finite ceiling for $Q$. Find the smallest $m>n$ such
that $Q\cap(M_{m}-M_{n})\neq\emptyset$ and choose $q\in Q\cap(M_{m}-M_{n})$;
then $q\notin M_{m-1}$, so we must have $q\leqslant p_{m}\in L^{u}$,
so $p_{m}\in Q$ as $Q$ is an upper subset of $L^{u}$. Let $r=p_{m}$.
By condition (b2) $U_{m}$ contains an element $B$ (say) with $t(B)=r\in Q-M_{m-1}$,
so $\widetilde{B}\cap A=\emptyset$ by (\ref{eq:Um}), whilst $\widetilde{B}\cap X_{r}\neq\emptyset$
as $\widetilde{B}$ is $r$-trim, and so $X_{r}\nsubseteq A$. Hence
$X_{Q}$ is unbounded, as required.

\textbf{Step 4: }$\mathscr{X}$ is a $[P,L,f]$-partition

We claim that $\overline{X_{p}}$ is compact iff $p\in L$. For if
$p\in L$, let $Q=\{q\in P\mid q\leqslant p\}$; then $Q$ satisfies
condition (\ref{eq:XQ compact}), as $p\in P_{\Delta}$ and either
$p\in L^{b}$, when $Q\cap L^{u}$ is empty (as $L^{b}$ is a lower
subset of $L$), or $p\in L^{u}$ and $Q$ has a finite ceiling. Hence
$X_{Q}$ is bounded and $\overline{X_{p}}$ is compact. Conversely,
suppose $p\notin L$. Then for all large enough $n$, by construction,
$U_{n}$ has an element $B_{n}$, say, of type $p$, and so $\widetilde{B_{n}}\cap X_{p}\neq\emptyset$,
whereas if $A\in R_{n-1}$ then $\widetilde{B_{n}}\cap A=\emptyset$
by (\ref{eq:Um}), so $X_{p}\nsubseteq A$. Hence there is no $A\in R$
such that $X_{p}\subseteq A$, and $\overline{X_{p}}$ is not compact.

If now $p\in L_{\min}^{d}$, then $Z(n)$ contains $f(p)$ elements
of type $p$ for all $n\geqslant m$, where $p=p_{m}$, as $p$ is
minimal in $P$ and so condition (b1) will bring in $f(p)$ points
of type $p$ precisely when $n=m$, condition (c) does not apply,
and condition (a) ensures that there continue to be exactly $f(p)$
such elements when $n>m$; hence by counting $p$-trim sets we have
$|\widetilde{Z(n)}\cap X_{p}|=f(p)$ for $n\geqslant m$. So $\mathscr{X}$
is indeed a trim $[P,L,f]$-partition of $W$.

Finally, if $W$ is compact, then by Proposition~\ref{Compact means FF}
$P$ has a finite foundation, and we must have $L=P$ and $L^{u}$
having a finite ceiling. The converse follows immediately by taking
$Q=P$ in (\ref{eq:XQ compact}) above.

\textbf{Step 5 }The partition $\mathscr{X}$ will be complete iff
$P$ has the ACC (Theorem~\ref{Trim nec condition})\@. If $P$
is $\omega$-complete, with $P^{d}$ separated in $P$ if $P$ is
a PO system, then by Corollary~\ref{Completion corollary 2} and
Proposition~(\ref{Prop: Ideals})(iii) $\mathscr{X}$ can be regularly
extended to a complete semi-trim $P$-partition $\mathscr{Y}=\{Y_{p}\mid p\in P\}$
of $W$ such that $X_{p}\subseteq Y_{p}$, with $Y_{p}\subseteq\overline{X_{p}}$
for all $p\in P$ (as if $y\in Y_{p}$ and $A\in V_{y}$, then $t(A)\leqslant p$,
so $A\cap X_{p}\neq\emptyset$). Hence $\overline{Y_{p}}=\overline{X_{p}}$
for all $p$, from which it follows easily that $\mathscr{Y}$ is
a $[P,L,f]$-partition of $W$ satisfying (i) and (ii).
\end{proof}
\begin{rem}
If $P$ has a finite foundation, then $P_{\min}$ and $L_{\min}^{d}$
will be finite.

It is interesting that the existence of complete trim, complete semi-trim,
or $P$-partitions depends on the properties of ascending chains in
$P$, whereas the existence of a compact Stone space that admits a
semi-trim $P$-partition depends on the properties of minimal elements
of $P$.

By Proposition~\ref{Isolated propn}, the isolated points of the
Stone space constructed in the previous Theorem will be the elements
of $X_{p}$ for $p\in P^{d}\cap P_{\max}$. So:

(a) $R$ will be atomless and the Stone space $W$ will be the Cantor
set (possibly minus one point) if $P^{d}\cap P_{\max}=\emptyset$; 

(b) if $P^{d}\cap P_{\max}\subseteq P_{\min}$, then $R$ will be
the direct product of a countable atomless Boolean ring (which may
or may not have an identity) and the ring of finite subsets of $P^{d}\cap P_{\max}$,
and the Stone space $W$ will be the Cantor set (possibly minus one
point) plus a set of isolated points, one for each $p\in P^{d}\cap P_{\max}$;

(c) otherwise $R$ and $W$ will have more complex structures.

For cases (a) and (b), if also $W$ is compact, we can embed $W$
as a closed subset of the unit interval $[0,1]$, and by letting $Q=P\cup\{\top\}$,
where $\top$ is an added top element, and letting $X_{\top}=[0,1]-W$,
which is open and dense in $[0,1]$, we obtain a trim $Q$-partition
of the unit interval $[0,1]$.
\end{rem}

\subsection{Existence results}

The next 3 results and their proofs relate to PO systems and follow
on directly from Theorem~\ref{existence construction}. If $(P,\vartriangleleft)$
is a poset, the relevant existence statement for the first two results
follows by considering the PO system $(\widetilde{P},<)$ which has
the same underlying set as $P$ and with $\widetilde{P}^{d}=\emptyset$:
i.e.\ $p<q$ iff $p\trianglelefteqslant q$.
\begin{thm}
\label{Existence theorem}Let $P$ be a countable poset or PO system.
Then:

(i) There is always an $\omega$-Stone space that admits a trim $P$-partition;

(ii) There is an $\omega$-Stone space that admits a complete trim
$P$-partition iff $P$ satisfies the ACC;

(iii) There is an $\omega$-Stone space that admits a complete semi-trim
$P$-partition iff $P$ is $\omega$-complete, with also $P^{d}$
separated in $P$ if $P$ is a PO system;

(iv) If $P$ is a poset, there is an $\omega$-Stone space that admits
a complete $P$-partition iff $P$ is $\omega$-complete.
\end{thm}

\begin{proof}
The existence of the relevant $P$-partitions in (i) to (iv) follows
from the construction in Theorem~\ref{existence construction}, taking
$L^{u}=\emptyset$, $L=L^{b}=P_{\Delta}$ and $f(p)=1$ for all $p\in L_{\min}^{d}$,
noting that $P_{\Delta}$ is a lower subset of $P$. The necessity
of the conditions on $P$ in parts (ii), (iii) and (iv) follows from
Theorem~\ref{Trim nec condition} and Proposition~\ref{Prop: Ideals}(ii)
and (iii).
\end{proof}
\begin{cor}
Let $P$ be a countable poset or PO system. Then there is a compact
Stone space with a trim $P$-partition iff $P$ has a finite foundation.
\end{cor}

\begin{proof}
``Only if'' was shown in Proposition~\ref{Compact means FF}. ``If''
follows from Theorem~\ref{existence construction}, taking $L^{u}=\emptyset$,
$L=L^{b}=P$ and $f(p)=1$ for all $p\in L_{\min}^{d}$.
\end{proof}
\begin{cor}
Let $P$ be a countable PO system such that $P=P_{\Delta}$, and $f:P_{\min}^{d}\rightarrow\mathbb{N}_{+}$.

(i) If $P$ has a finite ceiling, then any Stone space admitting a
trim $[P,P,f]$-partition is compact;

(ii) If $P$ does not have a finite ceiling, then there is a non-compact
Stone space admitting a trim $[P,P,f]$-partition.
\end{cor}

\begin{proof}
(i) If $G$ is a finite ceiling for $P$ and $\{X_{p}\mid p\in P\}^{*}$
is a trim $[P,P,f]$-partition of the Stone space $W$, then $\overline{X_{P}}=\bigcup_{p\in G}\overline{X_{p}}$
and so is compact; but $X_{P}$ is dense in $W$, so $W$ is compact. 

(ii) If $P$ does not have a finite ceiling, take $L=L^{u}=P$ in
Theorem~\ref{existence construction} to obtain a trim $[P,P,f]$-partition
of a non-compact Stone space.
\end{proof}

\section{\label{sec:Uniqueness-Conditions}Uniqueness conditions}

We now explore compactness and boundedness conditions which result
in unique trim $P$-partitions for a given countable PO system $P$.
If $P$ is a poset, we can then choose an arbitrary subset $P^{d}$
of $P$ and convert $P$ into a PO system to yield $P$-partitions
with the chosen partition elements being discrete and the rest having
no isolated points.
\begin{thm}
\label{uniqueness thm} (uniqueness) Let $[P,L,f]$ be a countable
extended PO system such that $L$ has a finite foundation. 

Then there is an $\omega$-Stone space $W$ and a bounded trim $[P,L,f]$-partition
$\mathscr{X}$ of $W$, and such a pair $\{W,\mathscr{X}\}$ is unique
up to $P$-homeomorphism.
\end{thm}

\begin{proof}
For existence, take $L^{b}=L$ and $L^{u}=\emptyset$ in the construction
in Theorem~\ref{existence construction}, noting that $L\subseteq P_{\Delta}$
and $X_{L}$ will be bounded. We must show uniqueness. Let $F$ be
a finite foundation for $L$; by Lemma~\ref{Finfound} we can take
$F$ to be $L\cap P_{\min}$ as $L$ is a lower set.

Suppose therefore that we have bounded trim $[P,L,f]$-partitions
$\mathscr{X}=\{X_{p}\}^{*}$ of $X$, with underlying countable Boolean
ring $R$ and type function $T$, and $\mathscr{Y}=\{Y_{p}\}^{*}$
of $Y$, with underlying countable Boolean ring $S$.

\textbf{Step 1.} Find $A\in R$ and $B\in S$ such that $X_{L}\subseteq A$
and $Y_{L}\subseteq B$. Using Proposition~\ref{Propn std split of trim set}(i),
write $A=A_{1}\dotplus\ldots\dotplus A_{n}$, where each $A_{n}$
is trim. Discarding any redundant $A_{j}$'s as necessary and reducing
$A$ accordingly, we may assume that $T(A_{j})\cap L\neq\emptyset$
for each $j$, so that $L\subseteq T(A)$ and $T(A)_{\min}=L_{\min}=F$.
Using Proposition~\ref{Propn std split of trim set}(i) again, we
can assume that the partition of $A$ contains precisely $n_{p}$
$p$-trim sets for each $p\in F$, where $n_{p}=1$ for $p\in F-P^{d}$
and $n_{p}=|A\cap X_{p}|=f(p)$ for $p\in F\cap P^{d}=L_{\min}^{d}$.
Repeating this process for $B$, we obtain matching trim partitions
$A=A_{1}\dotplus\ldots\dotplus A_{n}$ and $B=B_{1}\dotplus\ldots\dotplus B_{n}$
such that $t(A_{j})=t(B_{j})$ for each $j$, so that $\mu(A)=\mu(B)$.
By Theorem~\ref{Thm:same mu=00003Dhomeom}, taking $a\in X_{p}$
and $b\in Y_{p}$ for some $p\in F$, there is a $P$-homeomorphism
$\theta:A\rightarrow B$. 

\textbf{Step 2}. If $L=P$, then $R=(A)$ and $S=(B)$ and we are
done; otherwise $X$ and $Y$ are both non-compact. Suppose $D\in\Trim(\mathscr{X})$
with $D\cap A=\emptyset$, and $C\in S$ such that $B\subseteq C$,
and let $t(D)=q$. Then $q\notin L$, so $\overline{Y_{q}}$ is non-compact
and $Y_{q}\nsubseteq C$, so we can find $y\in Y_{q}-C$ and a $q$-trim
neighbourhood $E$ of $y$ such that $E\cap C=\emptyset$. By Theorem~\ref{Thm:same mu=00003Dhomeom},
$E$ is $P$-homeomorphic to $D$.

\textbf{Step 3}. If now $D\in R$ is any set with $D\cap A=\emptyset$,
and $C\in S$ with $B\subseteq C$, we can write $D$ as a disjoint
union of trim sets and apply step~2 successively to find $E\in S$
such that $E\cap C=\emptyset$ and $E$ is $P$-homeomorphic to~$D$.
A routine back-and-forth argument now extends the $P$-homeomorphism
$\theta:A\rightarrow B$ found in step 1 to a $P$-homeomorphism $\alpha:X\rightarrow Y$,
as required.
\end{proof}
\begin{cor}
\label{Finite case} Let $[P,L,f]$ be a finite extended PO system. 

Then there is an $\omega$-Stone space $W$ that admits a complete
trim {[}$P,L,f]$-partition, unique up to $P$-homeomorphism, and
$W$ will be compact iff $L=P$.
\end{cor}

\begin{proof}
This is immediate from Theorem~\ref{uniqueness thm}, as $L\subseteq P_{\varDelta}=P$
and has a finite foundation, any $[P,L,f]$-partition will be bounded,
and by STP3 a finite trim partition is complete.
\end{proof}
\begin{rem}
Corollary~\ref{Finite case} for compact $W$ follows from the work
of Pierce~\cite[Corollary 4.7]{Pierce}.
\end{rem}

\begin{cor}
\label{Trimuniqueness} Let $[P,L,f]$ be a countable extended PO
system such that $L$ has a finite ceiling. 

Then there is an $\omega$-Stone space $W$ with a trim $[P,L,f]$-partition
$\mathscr{X}$, unique up $P$-homeomorphism. Moreover $\mathscr{X}$
will be bounded, and will be complete iff $P$ has the ACC.
\end{cor}

\begin{proof}
The existence of $W$ and a trim $[P,L,f]$-partition $\{X_{p}\mid p\in P\}^{*}$
follows from Theorem~\ref{existence construction}, taking $L^{b}=L$,
and the ``complete iff ACC'' assertion from Theorem~\ref{Trim nec condition}. 

Let $G\subseteq L$ be a finite ceiling for $L$. By combining the
finite foundations for each element of $G$, noting that $G\subseteq P_{\Delta}$,
we will obtain a finite foundation for $L$. Moreover, by Remark~\ref{rem:Existence remark}
any trim $[P,L,f]$-partition is bounded. The uniqueness of a Stone
space with a trim $[P,L,f]$-partition now follows from Theorem~\ref{uniqueness thm}.
\end{proof}
For the next Corollary, if $P$ is a PO system we will call a trim
partition $\mathscr{X}$ of a Stone space a \emph{unitary trim $P$-partition
}if $\mathscr{X}$ is a trim $[P,L,f]$-partition for some $L\supseteq P_{\min}^{d}$,
with $f(p)=1$ for $p\in P_{\min}^{d}$. 
\begin{cor}
\label{No min elts corollary}Let $P$ be a countable PO system. 

(i) If $P_{\Delta}=P_{\min}^{d}$ and $P_{\Delta}$ is finite (in
particular, if $P$ has no minimal elements), then up to $P$-homeomorphism
there is a unique $\omega$-Stone space with a unitary trim $P$-partition;

(ii) if $P_{\Delta}\neq P_{\min}^{d}$, then up to $P$-homeomorphism
there are multiple $\omega$-Stone spaces admitting a unitary trim
$P$-partition.
\end{cor}

\begin{proof}
(i) Write $Q=P_{\Delta}$. The conditions mean that $\overline{X_{p}}$
is compact (indeed, singleton) iff $p\in Q$ (Proposition~\ref{Compact means FF}),
so $X_{Q}$ is finite and hence compact, and the partition is a bounded
trim $[P,Q,f]$-partition. Hence by Theorem~\ref{uniqueness thm}
there is a unique $\omega$-Stone space with a trim $[P,Q,f]$-partition.

(ii) If $P_{\Delta}\neq P_{\min}^{d}$, apply Theorem~\ref{existence construction}
first with $L=L^{b}=P_{\min}^{d}$ and second with $L=L^{b}=P_{\Delta}$,
taking $L^{u}=\emptyset$ and noting that $P_{\min}^{d}\subseteq P_{\Delta}$,
to obtain two $\omega$-Stone spaces with non-$P$-homeomorphic trim
$P$-partitions.
\end{proof}
\begin{rem}
The proof of Theorem~\ref{uniqueness thm} can be adapted to show
that uniqueness also applies in the previous Corollary when $P_{\Delta}=P_{\min}^{d}$
and $P_{\Delta}$ is infinite.
\end{rem}

The next Corollary shows that automorphisms of the extended PO system
can be extended to $P$-homeomorphisms of the Stone space.
\begin{cor}
Let $[P,L,f]$ be a countable extended PO system and $\mathscr{X}=\{X_{p}\mid p\in P\}^{*}$
a bounded trim $[P,L,f]$-partition of the $\omega$-Stone space $W$. 

Let $\theta$ be a $[P,L,f]$-automorphism. Then there is a $P$-homeomorphism
$\beta$ of $W$ such that $X_{p}\beta=X_{p\theta}$ for all $p\in P$.
\end{cor}

\begin{proof}
Let $R$ be the ring of compact opens in $W$ and let $Y_{p}=X_{p\theta}$
for each $p\in P$. It is easy to see that $\mathscr{Y}=\{Y_{p}\mid p\in P\}^{*}$
is also a bounded trim $[P,L,f]$-partition of $W$; in particular,
for $A\in R$, $A\cap X_{p\theta}=A\cap Y_{p}$, and $\Trim_{p\theta}(\mathscr{X})=\Trim_{p}(\mathscr{Y})$,
as $A\cap X_{p\theta}$ is a singleton iff $A\cap Y_{p}$ is. So by
Theorem~\ref{uniqueness thm} there is a $P$-homeomorphism $\beta$
of $W$ such that $X_{p}\beta=X_{p\theta}$.
\end{proof}
We now apply previous results to generate two interesting examples.
\begin{cor}
Let $P$ be a PO system whose underlying poset is the (unique) homogeneous
poset which embeds any finite poset, and let $W$ be the Cantor set
minus a point. 

Then there is a trim $P$-partition of $W$, unique up to $P$-homeomorphism.

In particular, there are unique trim $P$-partitions of $W$ for which
each partition element (a) has no isolated points, or (b) is discrete.
\end{cor}

\begin{proof}
Existence follows from Theorem~\ref{existence construction} with
$L=\emptyset$, noting that $P$ has no minimal elements and $P_{\Delta}=\emptyset$.
The $\omega$-Stone space will be non-compact, and without isolated
points as $P_{\max}=\emptyset$ (Proposition~\ref{Isolated propn}).
So the underlying Boolean ring will be countable atomless without
identity, and the Stone space will be the Cantor set minus a point.
Uniqueness follows from Corollary~\ref{No min elts corollary}(i).
For the final assertion, take $P^{d}=\emptyset$ or $P^{d}=P$.
\end{proof}
\begin{cor}
Let $\mathbb{\mathrm{\mathit{P}=}Q\cap\mathrm{[0,1]}}$ and $\mathbb{\mathrm{[0,1]}}$
be the totally ordered subsets of $\mathbb{Q}$ and $\mathbb{R}$
respectively consisting of points in the unit interval. Then:

(i) there is a trim $P$-partition of the Cantor set such that each
partition element has no isolated points, unique up to $P$-homeomorphism;

(ii) there is a complete semi-trim partition $\{X_{r}\mid r\in[0,1]\}$
of the Cantor set $W$, with $\overline{X_{r}}=\bigcup_{s\leqslant r}X_{s}\text{ for all }r\in[0,1]$
and $X_{0}$ homeomorphic to $W$;

(iii) there is a complete semi-trim $[0,1]$-partition of the unit
interval.
\end{cor}

\begin{proof}
(i) Let $\widetilde{P}$ be the PO system with underlying poset $P$
such that $P^{d}=\emptyset$. $\widetilde{P}$ has a finite foundation
$\{0\}$, so by Theorem~\ref{uniqueness thm} (taking $L=\widetilde{P}$
and noting that $L_{\min}^{d}=\emptyset$), there is a Stone space
$W$ of a countable Boolean ring $R$ admitting a trim $\widetilde{P}$-partition
such that $W$ is compact and each partition element has no isolated
points, which is unique up to $P$-homeomorphism. $R$ will therefore
be atomless with identity, and so $W$ will be homeomorphic to the
Cantor set. 

Part (ii) follows from part (i) and Corollary~\ref{Completion corollary},
as $P^{d}=\emptyset$ and $[0,1]$ is the chain completion of $P$
and $P=\widehat{P}$. As $X_{0}$ is closed and compact with no isolated
points, its compact open subsets form an atomless Boolean ring, and
hence $X_{0}$ is homeomorphic to the Cantor set.

For (iii), let $\{X_{r}\mid r\in[0,1]\}$ be the partition of the
Cantor set~$W$ from (ii). Embed $W$ as a closed subset of $[0,1]$,
and replace $X_{1}$ with $X_{1}\cup([0,1]-W)$, so that $\overline{X_{1}}=[0,1]$.
\end{proof}
\begin{rem}
By a similar argument, we can find a complete semi-trim $\mathbb{R}$-partition
of the Cantor set minus a point.
\end{rem}

\section{\label{sec:Closure-algebras-and}Closure algebras and the Cantor
set}

In this final section we look at partitions associated with closure
algebras that are generated by a single open set within a topological
space.

We recall that a \emph{closure algebra }is a Boolean algebra $S$
together with an additional closure operator $^{C}$ such that $x\subseteq x^{C}$,
$x^{CC}=x^{C}$, $(x\cup y)^{C}=x^{C}\cup y^{C}$ (all $x,y\in S$),
and $0^{C}=0$. We write $\overline{x}$ for $x^{C}$. (For our purposes,
it is convenient to work with closure algebras rather than interior
algebras.) 

For any topological space $W$, we can regard $2^{W}$ as a closure
algebra with the usual set operations of $+$ and $\cap$, together
with the closure operator which maps any subset $A$ of $W$ to $\overline{A}$,
the smallest closed set containing $A$.
\begin{notation*}
For a poset $P$, let $\mathscr{E\mathrm{(}\mathit{P\mathrm{)}}}$
be the closure sub-algebra of $2^{P}$ generated by the closed (i.e.\ lower)
subsets of $P$, where $P$ is given the usual poset topology, with
closure operator $Q^{C}=Q_{\downarrow}$ for $Q\subseteq P$.

Let $P(\infty)$ denote the poset with elements $\{p_{k}\mid k\geqslant0\}$,
and relations $p_{k}\lneqq p_{j}$ iff $k\geqslant j+2$. Let $\widetilde{P(\infty)}=P(\infty)\cup\{\bot\}$,
having the relations of $P(\infty)$ together with extra relations
$\bot\lneqq p_{k}$ for all $k\geqslant0$ for the new minimum element
$\bot$. It can be readily verified that the Hasse diagram of $P(\infty)$
is the well-known Rieger-Nishimura ladder.

Let $J=\mathbb{N}\times\{0,2\}$, and for $m\geqslant0$ let $P(m,0)=\{p_{0},p_{1},\ldots,p_{m}\}$
and $P(m,2)=P(m,0)\cup\{m+2\}$, each with the inherited relations
from $P(\infty)$.

If $A$ is a subset of the topological space $W$, let $A_{\infty}=W-\bigcup\{B\mid B\in\At(\mathscr{C})\}$,
where $\mathscr{C}$ is the closure sub-algebra of $2^{W}$ generated
by $A$ and $\At(\mathscr{C})$ denotes the atoms of $\mathscr{C}$. 
\end{notation*}
\begin{lem}
\label{1gen lemma}Let the poset $P$ be equal to either $P(\infty)$
or $P(m,n)$ for some $(m,n)\in J$. Then the closure algebra $\mathscr{E\mathrm{(}\mathit{P\mathrm{)}}}$
is generated by the single open set $\{p_{0}\}$, and consists of
all finite and cofinite subsets of $P$.
\end{lem}

\begin{proof}
For any $p\in P$, $\{q\in P\mid q\leqslant p\}$ is cofinite in $P$,
so all closed sets are cofinite in $P$, and $\mathscr{E\mathrm{(}\mathit{P\mathrm{)}}}$
viewed as a Boolean algebra is a subalgebra of the set of all finite
and cofinite subsets of $P$. The assertions now follow easily by
induction from the identity 
\[
\{p_{k+1}\}=P-\{p\in P\mid p\geqslant p_{k}\}-\{p\in P\mid p\leqslant p_{k}\}.
\]
\end{proof}
Blok~\cite{Blok} showed that there are continuously many non-isomorphic
interior (closure) algebras generated by one element. If we require
the generating element to be open, however, the picture changes dramatically:
it follows from the duality between Heyting algebras and closure algebras
generated by their open elements (\cite{McK-Tarski}) that such a
one-generator closure algebra corresponds to a quotient of the free
Heyting algebra $H$ on one generator. Now $H$ is just the Rieger-Nishimura
lattice, which is dually isomorphic to $\mathscr{E\mathrm{(}\mathit{P\mathrm{(\infty))}}}$,
and it can be verified that quotients of $H$ are dually isomorphic
to the $\mathscr{E\mathrm{(}\mathit{P\mathrm{(\mathit{m,n}))}}}$
for $(m,n)\in J$. 

The next theorem extends this by showing that a closure algebra generated
by an open set within a topological space is underpinned by a complete
trim partition of the space.
\begin{thm}
\label{Closed sets thm} Let $W$ be a topological space, $A$ a non-empty
open subset of $W$, and $\mathscr{C}$ the closure sub-algebra of
$2^{W}$ generated by $A$. Then $\mathscr{C}\cong\mathscr{E\mathit{\mathrm{(}}\mathit{P\mathit{\mathrm{)}}}}$,
where $P=P(\infty)$ or $P=P(m,n)$ for some $(m,n)\in J$, and under
this isomorphism:

(i) if $A_{\infty}=\emptyset$, $\At(\mathscr{C})$ is a complete
trim $P$-partition of $W$; 

(ii) if $A_{\infty}\neq\emptyset$, $\At(\mathscr{C})\cup\{A_{\infty}\}$
is a complete trim $\widetilde{P(\infty)}$-partition of $W$.
\end{thm}

\begin{proof}
Define open sets $U_{n}$ and $V_{n}$ and sets $B_{n}$ inductively
as follows, by duality with the Rieger-Nishimura lattice (Figure~\ref{fig:The-Rieger-Nishimura-lattice}).
Let $V_{-1}=V_{0}=\emptyset$, $U_{0}=B_{0}=A$, and for $n\geqslant1$
let:

\begin{align*}
U_{n} & =W-\overline{B_{n-1}}\text{ (i.e. \ensuremath{U_{n-1}\rightarrow V_{n-2}} in Heyting algebra notation)}\\
V_{n} & =U_{n-1}\cup V_{n-1}\\
B_{n} & =U_{n}-V_{n-1}
\end{align*}
\begin{figure}
\includegraphics[bb=100bp 150bp 800bp 540bp,clip,scale=0.4]{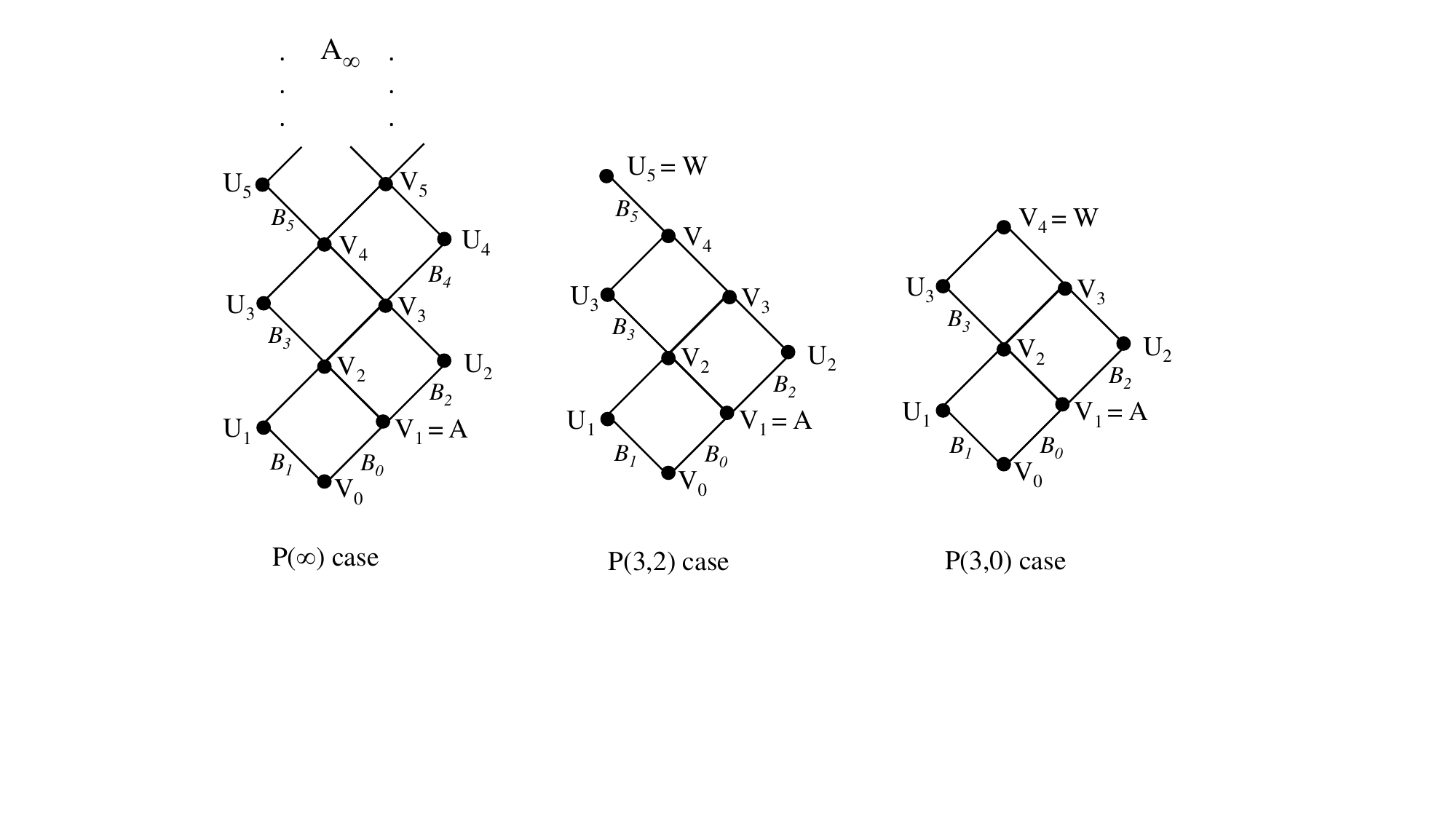}

\caption{\label{fig:The-Rieger-Nishimura-lattice}The Rieger-Nishimura lattice
for the closure algebra generated by an open set $A$}
\end{figure}

Induction yields the following, which can also be read off from Figure~\ref{fig:The-Rieger-Nishimura-lattice}:

\begin{align*}
V_{n-1} & =U_{n+1}\cap U_{n}\\
V_{n} & =U_{n+1}\cap V_{n+1}\\
B_{n} & =V_{n+1}-V_{n}
\end{align*}

It follows that the $B_{n}$ are disjoint, with

\begin{align}
U_{n} & =B_{n}\cup\bigcup_{j\leqslant n-2}B_{j}\nonumber \\
\overline{B_{n}} & =W-U_{n+1}=W-B_{n+1}-\bigcup_{j\leqslant n-1}B_{j}\label{eq:atomclosure-1}
\end{align}

Let $K=\{k\geqslant0\mid B_{k}\neq\emptyset\}$, $\mathscr{B}=\{B_{k}\mid k\in K\}$
and $P=\{p_{k}\in P(\infty)\mid k\in K\}$. We see that $\mathscr{B}=\At(\mathscr{C})$
and that elements of $\mathscr{C}$ are precisely the unions of finite
and cofinite subsets of $\mathscr{B}$. Further, the map sending each
$B_{k}$ to $p_{k}$ extends to a closure algebra isomorphism between
$\mathscr{C}$ and $\mathscr{E\mathit{\mathrm{(}}\mathit{P\mathit{\mathrm{)}}}}$,
as for $k\neq j$, $B_{k}\subseteq\overline{B_{j}}$ iff $k\geqslant j+2$
(by~(\ref{eq:atomclosure-1})) iff $p_{k}\lneqq p_{j}$, and $\mathscr{E\mathit{\mathrm{(}}\mathit{P\mathit{\mathrm{)}}}}$
consists of all finite and cofinite subsets of $P$ by Lemma~\ref{1gen lemma}. 

Now $A_{\infty}=W-\bigcup_{k\in K}B_{k}=W-\bigcup_{k\geqslant1}V_{k}$;
so $A_{\infty}$ is closed but is not an element of $\mathscr{C}$
unless it is empty. Hence $\mathscr{B}$, with the addition of $A_{\infty}$
if it is non-empty, is a partition of $W$. Assign types in $\widetilde{P(\infty)}$
to the elements of this partition via $t(B_{k})=p_{k}$ and $t(A_{\infty})=\bot$. 

Let $N=\min\{\infty,\{n\geqslant0\mid B_{n+1}=\emptyset\}\}$; we
note that $B_{0}=A\neq\emptyset$, and that if $B_{N+1}=\emptyset$,
then $B_{k}=\emptyset$ for all $k\geqslant N+3$ by~(\ref{eq:atomclosure-1}).
There are now four cases to consider. 

Case 1: $N$ is finite and $B_{N+2}=\emptyset$. Then $P=P(N,0)$,
$A_{\infty}=\emptyset$, and the atoms of $\mathscr{C}$ are $\{B_{0},B_{1},\ldots,B_{N}\}$
and form a complete $P$-partition of $W$. 

Case 2: $N$ is finite and $B_{N+2}\neq\emptyset$. Then $P=P(N,2)$,
$A_{\infty}=\emptyset$, and $\At(\mathscr{C})=\{B_{0},B_{1},\ldots,B_{N},B_{N+2}\}$,
which is a complete $P$-partition of $W$. 

Case 3: $N=\infty$ and $A_{\infty}=\emptyset$. Then $P=P(\infty)$
and the atoms of $\mathscr{C}$ are $\{B_{k}\mid k\geqslant0\}$ which
by~(\ref{eq:atomclosure-1}) form a complete $P$-partition of $W$.

Case 4: $N=\infty$ and $A_{\infty}\neq\emptyset$. The atoms of $\mathscr{C}$
are $\mathscr{B}=\{B_{k}\mid k\geqslant0\}$, and $\mathscr{B}\cup\{A_{\infty}\}$
is a complete $\widetilde{P(\infty)}$-partition of $W$, as $A_{\infty}$
is closed and $A_{\infty}\subseteq\overline{B_{k}}$ for all $k$
by~(\ref{eq:atomclosure-1}). 

Finally, we note that $\widetilde{P(\infty)}$ satisfies the ACC and
has no infinite antichains (indeed, all antichains have length at
most 2), and therefore for any subset $P$ of $\widetilde{P(\infty)}$,
a complete $P$-partition will also be a complete trim $P$-partition
by Proposition~\ref{basic properties}(iii).
\end{proof}
The previous theorem showed that the closure algebra generated by
any open subset of a topological space is isomorphic to $\mathscr{E\mathit{\mathrm{(}}\mathit{P\mathit{\mathrm{)}}}}$
for a suitable subset $P$ of $P(\infty)$. Restricting to the Cantor
set, we now show that each possible suitable $P$ can arise, and that
we can arrange for individual elements of the resulting partition
to be discrete or non-discrete at will. In the following Theorem,
let $\widetilde{P}=P$ in the case when $P$ is a finite poset. 
\begin{thm}
Let $W$ be the Cantor set and let $P=P(\infty)$ or $P=P(m,n)$ for
some $(m,n)\in J$. Let $I\subseteq\widetilde{P}-\widetilde{P}_{\max}$,
and let $Q$ be the PO system with underlying poset $\widetilde{P}$
such that $Q^{d}=I$. Then there is an open subset $A(Q)$ of $W$,
unique up to homeomorphisms of $W$, such that 

(i) the closure algebra $\mathscr{C}$ generated by $A(Q)$ is isomorphic
to $\mathscr{E\mathit{\mathrm{(}}\mathit{P\mathit{\mathrm{)}}}}$; 

(ii) the atoms of $\mathscr{C}$ under this isomorphism, together
with $A(Q)_{\infty}$ in the case that $P=P(\infty)$, form a complete
trim $Q$-partition of $W$, with any discrete atoms of $\mathscr{C}$
being singletons.
\end{thm}

\begin{proof}
Apply Corollary~\ref{Trimuniqueness} with $L=Q$ and $f(p)=1$ for
$p\in L_{\min}^{d}$, noting that $Q$ has a finite foundation, satisfies
the ACC and has a finite ceiling, to obtain a unique compact Stone
space $X$ of a countable Boolean ring $R$ with a complete trim $Q$-partition
$\{X_{p}\mid p\in\widetilde{P}\}$ of $X$ such that discrete partition
elements are singletons. As $X$ is compact and has no isolated points
(using Proposition~\ref{Isolated propn}, as $Q^{d}\cap Q_{\max}=\emptyset$),
$R$ is countable atomless with an identity, and $X$ is therefore
the Cantor set $W$.

Let $A(Q)=X_{p_{0}}$, which is open in $X$, and let $\mathscr{C}$
and $\mathscr{D}$ be the closure subalgebras of $2^{X}$ generated
by $A(Q)$ and by $\{X_{p}\mid p\in P\}$ respectively. 

Now $\mathscr{E\mathit{\mathrm{(}}\mathit{P\mathit{\mathrm{)}}}}$
viewed as a Boolean algebra is generated by the singleton subsets
of~$P$, so we can define a Boolean algebra map $\alpha:$~$\mathscr{E\mathit{\mathrm{(}}\mathit{P\mathit{\mathrm{)}}}}\rightarrow\mathscr{D}:\{s\}\mapsto X_{s}$
for each $s\in P$. If $P$ is finite, then $\alpha$ is clearly an
isomorphism of closure algebras, as $\overline{X_{s}}=\bigcup_{t\leqslant s}X_{t}$.
If $P$ is infinite, then $\overline{X_{s}}=X-\bigcup_{t\nleqslant s}X_{t}$,
with the latter union being finite, and $\{t\in P\mid t\leqslant s\}=P-\{t\in P\mid t\nleqslant s\}$,
so again $\alpha$ is an isomorphism of closure algebras.

By Lemma~\ref{1gen lemma}, $\mathscr{E\mathit{\mathrm{(}}\mathit{P\mathit{\mathrm{)}}}}$
is generated by $\{p_{0}\}$, and so $\mathscr{D}$ is generated by
$A(Q)$; hence $\mathscr{D}=\mathscr{C}$ and $\mathscr{C}$ is isomorphic
to $\mathscr{E\mathit{\mathrm{(}}\mathit{P\mathit{\mathrm{)}}}}$,
with atoms $\{X_{p}\mid p\in P\}$. Part~(ii) follows by observing
that $P=\widetilde{P}$ if $P$ is finite, and that $A(Q)_{\infty}=X_{\bot}\neq\emptyset$
if $P=P(\infty)$. 
\end{proof}
\begin{rem}
Let $P=P(m,0)$ in this Theorem with $m\geqslant3$, so that $P_{\min}=\{p_{m-1},p_{m}\}$.
By Proposition~\ref{Propn std split of trim set}(i), we can write
$X=X_{1}\dotplus X_{2}$ and $A(Q)=A(Q)_{1}\dotplus A(Q)_{2}$, where
$X_{i}$ is $p_{m-2+i}$-trim and $A(Q)_{i}=A(Q)\cap X_{i}$ ($i=1,2$).
Writing $Q_{1}=Q-\{p_{m-2},p_{m}\}$ and $Q_{2}=Q-\{p_{m-1}\}$, we
see that $X$ can be expressed as a disjoint union of two clopen subsets,
with the corresponding $A(Q)_{i}$ generating a complete trim $Q_{i}$-partition
of $X_{i}$, where the underlying poset of $Q_{i}$ is $P(m-4+i,2)$.
Likewise, a $P(1,0)$-partition splits $X$ into two clopen subsets,
while a $P(2,0)$-partition splits $X$ into a clopen subset and a
$P(0,2)$-partition. Hence the true building blocks amongst the $P(m,n)$
are $P(0,0)$ (which just partitions $X$ into itself) and $\{P(m,2)\mid m\geqslant0\}$.
\begin{rem}
It is natural to ask to what extent these results extend to closure
algebras generated by multiple open subsets of (for instance) the
Cantor set. By a result of Dobbertin~\cite{Dobbertin}, any Heyting
algebra can be embedded in the Heyting algebra of ideals of an atomless
Boolean algebra. Hence (by duality) for any finitely generated closure
algebra $\mathscr{D}$, we can find open subsets $A_{1},\ldots,A_{n}$
of the Cantor set which generate a closure algebra isomorphic to $\mathscr{D}$.
The position with respect to underlying partitions however is much
less clear$\ldots$
\end{rem}

\end{rem}

\begin{question}
If $A_{1},\ldots,A_{n}$ are open subsets of the Cantor set with $n\geqslant2$,
can we find a ``nice'' (e.g.\ a complete semi-trim) partition that
refines the atoms of the Boolean algebra $\langle C_{1},\ldots,C_{n}\rangle$?
\end{question}

\paragraph{Declarations:}
\begin{itemize}
\item \textbf{Funding}: no funding was received for conducting this study
or to assist with the preparation of this manuscript. 
\item \textbf{Conflicts of interest}: the author has no relevant financial
or non-financial interests to disclose. 
\item \textbf{Availability of data and material}: data sharing is not applicable
to this article as no datasets were generated or analysed during the
current study. 
\item \textbf{Code availability}: not applicable 
\end{itemize}

\end{document}